\documentclass[11pt, oneside, a4paper]{article}
\usepackage{amsmath}
\usepackage{latexsym}
\usepackage{amsfonts}
\usepackage[margin=30 mm]{geometry}
\usepackage{amssymb}
\usepackage{amsthm}
\usepackage{bbm}
\usepackage{etoolbox}
\usepackage{yfonts}
\usepackage{IEEEtrantools}
\usepackage{mathrsfs}
\usepackage[nottoc]{tocbibind}
\usepackage{tikz-cd}
\usepackage{hyperref}
\author{Yao Ming Chan}
\title{Wedge product matrices and orbits of principal congruence subgroups}
\theoremstyle{definition}
\newtheorem{definition}{Definition}[section]

\theoremstyle{definition}
\newtheorem{example}{Example}[section]

\theoremstyle{theorem}
\newtheorem{theorem}{Theorem}[section]
\newtheorem{lemma}[theorem]{Lemma}
\newtheorem{remark}[theorem]{Remark}
\newtheorem{corollary}[theorem]{Corollary}
\newtheorem{proposition}[theorem]{Proposition}
\numberwithin{equation}{section}
\AtEndEnvironment{example}{\qed}

\providecommand{\keywords}[1]
{	
  \textbf{\textit{Keywords---}} #1
}

\begin{document}
\date{}
\maketitle

\begin{abstract} 
Following Bump and Hoffstein in \cite{BH86}, the orbits in $\Gamma_{\infty}(3) \backslash \Gamma(3)$ are in bijection with sets of invariants satisfying certain relations. We explain how wedge product matrices give an alternative definition of the invariants of matrix orbits. This new method provides the possibility of performing similar computations with other congruence subgroups and arbitrary $n \times n$ matrices. Using Steinberg's refined version of the Bruhat decomposition, we construct an explicit choice of coset representative for each orbit in the orbit space $\Gamma_{\infty}(3) \backslash \Gamma(3)$ of $3 \times 3$ matrices over the PID of Eisenstein integers.
\end{abstract}



\tableofcontents
\section*{Introduction}
\addcontentsline{toc}{section}{\protect\numberline{}Introduction}

As explained in Bump and Hoffstein's paper \cite[Section 0]{BH86}, the Jacobi theta function is a constant multiple of the residue of the Eisenstein series 

\begin{equation} \label{equation: Eisenstein 2}
E(z, s) = y^s \sum_{\gamma \in \Gamma_{\infty}(2) \backslash \Gamma(2)} j(\gamma, z)^{-1} \lvert c z + d \rvert^{-2s}
\end{equation}
at the simple pole $s = \frac{1}{2}$, where 
\begin{equation*}
\Gamma(2) = \Big\{
\begin{pmatrix}
a & b \\
c & d 
\end{pmatrix} \in SL_2(\mathbb{Z}) \; \vert \; 
a \equiv d \equiv 1 \text{ mod } 2, \; b \equiv c \equiv 0 \text{ mod } 2
 \Big\}
\end{equation*}
and $\Gamma_{\infty}(2)$ is the subgroup of upper triangular matrices in $\Gamma(2)$. In \cite{BH86}, Bump and Hoffstein study the minimal parabolic Eisenstein series 

\begin{equation} \label{equation: Eisenstein 3 Minimal}
E_{\nu_1, \nu_2}(\tau) = \sum_{\gamma \in \Gamma_{\infty}(3) \backslash \Gamma(3)} \kappa(\gamma) I_{\nu_1, \nu_2}(\gamma \tau),
\end{equation}
where

\begin{equation} \label{equation: Gamma 3 Intro}
\Gamma(3) = \left\{ A \in SL_3(\mathbb{Z}[e^{2 \pi i/3}]) \; \Big\vert \; A \equiv 
\begin{pmatrix}
1 & & \\
& 1 & \\
& & 1
\end{pmatrix} \text{ mod } 3 \mathbb{Z}[e^{2 \pi i/3}] \right\},
\end{equation}
$\Gamma_{\infty}(3)$ is the subgroup of upper triangular unipotent matrices in $\Gamma(3)$, $\kappa$ is the Kubota symbol defined in \cite[Equation 1.5]{BH86}, $\tau$ is a matrix of the form

\begin{equation*}
\tau = 
\begin{pmatrix}
1 & x_2 & x_3 \\
 & 1 & x_1 \\
 &  & 1 
\end{pmatrix}
\begin{pmatrix}
y_1 y_2 & & \\
& y_1 & \\
& & 1
\end{pmatrix}
\end{equation*}
with $x_1, x_2, x_3 \in \mathbb{C}$ and $y_1, y_2 \in \mathbb{R}_{> 0}$ and 

\begin{equation*}
I_{\nu_1, \nu_2}(\tau) = y_1^{2 \nu_1 + \nu_2} y_2^{\nu_1 + 2 \nu_2}. 
\end{equation*}
The congruence condition in equation \eqref{equation: Gamma 3 Intro} is computed entry by entry. By taking the residue at either $\nu_1 = \frac{8}{9}$ or $\nu_2 = \frac{8}{9}$, Bump and Hoffstein obtain a maximal parabolic Eisenstein series whose Fourier coefficients are cubic $L$-functions. \newline

The computation of these Fourier coefficients depends on a specific choice of coset representatives in the orbit space $\Gamma_{\infty}(3) \backslash \Gamma(3)$ of principal congruence subgroups. There is a bijection between cosets in $\Gamma_{\infty}(3) \backslash \Gamma(3)$ and sets of invariants $A_1, B_1, C_1, A_2, B_2, C_2 \in \mathbb{Z}[e^{2 \pi i/3}]$ satisfying the conditions outlined in \eqref{equation: I1} \eqref{equation: I2}, \eqref{equation: I3} and \eqref{equation: I4} (see \cite[Lemma 3]{Pro85}). The main theorem we prove in this paper is Theorem \ref{theorem: Bruhat as Invariants + Cases}, which demonstrates how, from a set of invariants $A_1, B_1, C_1, A_2, B_2, C_2 \in \mathbb{Z}[e^{2 \pi i/3}]$ satisfying conditions \eqref{equation: I1}, \eqref{equation: I2}, \eqref{equation: I3} and \eqref{equation: I4}, one constructs an explicit matrix representative of an orbit in $\Gamma_{\infty}(3) \backslash \Gamma(3)$ whose invariants are exactly $A_1, B_1, C_1, A_2, B_2, C_2$. We remark that the construction of matrix representatives in $\Gamma_{\infty}(3) \backslash \Gamma(3)$ was briefly treated by Proskurin in \cite[Proposition 1.2.3]{Pro98}. Theorem \ref{theorem: Bruhat as Invariants + Cases} divides naturally into five cases: 

\begin{equation} \label{equation: Cases} \tag{Cases}
\begin{matrix}
\textbf{Case 1:} & \text{$A_1 \neq 0$ and $A_2 \neq 0$} \\
\textbf{Case 2:} & \text{$A_1 \neq 0$ and $A_2 = 0$} \\
\textbf{Case 3:} & \text{$A_1 = 0$, $B_1 \neq 0$ and $A_2 \neq 0$} \\
\textbf{Case 4:} & \text{$A_1 = 0$, $B_1 \neq 0$ and $A_2 = 0$} \\
\textbf{Case 5:} & \text{$A_1 = B_1 = A_2 = 0$}
\end{matrix}
\end{equation}

Theorem \ref{theorem: Bruhat as Invariants + Cases} provides one possible explicit bijection from the orbits in $\Gamma_{\infty}(3) \backslash \Gamma(3)$ to sets of invariants satisfying \eqref{equation: I1}, \eqref{equation: I2}, \eqref{equation: I3} and \eqref{equation: I4}. In \cite[Theorem 5.4]{Bum84}, Bump proved that a similar bijection holds for the orbit space $G_{\infty} \backslash G$, where $G = SL_3(\mathbb{R})$ and $G_{\infty}$ is the subgroup of upper triangular unipotent matrices in $G$. By using the Bruhat decomposition in \cite[Theorem 4]{Ste67}, canonical representatives for each orbit in $G_{\infty} \backslash G$ were also provided by Bump in \cite[Equations 5.8-5.13]{Bum84}.  \newline

The proof of Theorem \ref{theorem: Bruhat as Invariants + Cases} relies on a similar technique that Steinberg used to prove a variant of the Bruhat decomposition in \cite[Theorem 15]{Ste67}. We will refer to this technique as \emph{Steinberg reduction}. As described in \cite{Lus10}, the Bruhat decomposition allows one to reduce questions about reductive algebraic groups to questions about their Weyl groups. This is useful to further understand the structure of algebraic groups and their representations. The Bruhat decomposition is ubiquitous and has been proved in a variety of different contexts. For example, see \cite{Ren86}, \cite{HMP86} and \cite{Gal21}. \newline

In \cite{BH86}, the invariants of $A \in \Gamma(3)$ are the bottom rows of $A$ and $^{\iota}A$ where $^{\iota}: \Gamma(3) \rightarrow \Gamma(3)$ is an involution. In this paper, we use a new method of defining the invariants with the \emph{wedge product matrix} $\Lambda^2(A)$ (also known as a compound matrix c.f. \cite{Mul98}). We will explain that our approach to the invariants is another point of view to the approach adopted in \cite{BH86} which also opens up the path to handling more general cases. The heart of the matter is that if $A$ is a $n \times n$ matrix with $n > 3$ then the matrices $A$ and $^{\iota}A$ do not yield all of the invariants of $A$; instead the wedge product matrices $\Lambda^k(A)$ with $k \in \{1, 2, \dots, n-1 \}$ supply the desired invariants of $A$. \newline

\textbf{Acknowledgements.} This paper is based on the final chapter of my masters thesis, written at the University of Melbourne. I express my thanks to Arun Ram for his supervision and guidance through the process of writing my masters thesis and this paper. I also thank my family for their constant support and the anonymous reviewer for their instructive feedback on an earlier draft of this paper. 

\section{Steinberg reduction} 

Let $R$ be a principal ideal domain and $M_{m \times n}(R)$ denote the ring of $m \times n$ matrices with elements in $R$. \emph{Steinberg reduction} is an algorithm which begins with a matrix $(a_1, \dots, a_n)^T \in M_{n \times 1}(R)$ and outputs a specific matrix $A \in GL_n(R)$ such that 

\begin{equation*}
A
\begin{pmatrix}
a_1 \\
a_2 \\
\vdots \\
a_n
\end{pmatrix} = 
\begin{pmatrix}
\gcd(a_1, a_2, \dots, a_n) \\
0 \\
\vdots \\
0
\end{pmatrix}.
\end{equation*}

Steinberg reduction hinges on the set

\begin{equation} \label{equation: YR}
Y(R) = \Big\{
\begin{pmatrix}
a & b \\
c & d \\
\end{pmatrix} \in SL_2(R) \; \Big \vert \; 
\begin{matrix}
c \in (R - \{ 0 \})/R^{\times}, \\
a \in R/c R
\end{matrix}
 \Big\} \cup \{ I_2 \} \subseteq SL_2(R)
\end{equation}
where $I_2$ denotes the $2 \times 2$ identity matrix. Note that throughout this paper, we refer to the representatives of each equivalence class in $R/cR$ as elements of $R$.  The group $R^{\times}$ of units in $R$ acts on $R$ via multiplication. The set of $R^{\times}$-orbits $R/R^{\times}$ consists of representatives of the ideals of $R$. In an abuse of notation, if $a_1, \dots, a_n \in R$ then we will refer to $\gcd(a_1, \dots, a_n)$ as both a representative of an ideal in $R/R^{\times}$ and as an element of $R$, up to multiplication by a unit. 

\begin{example}
Let $\omega = e^{2 \pi i/3}$ and $\textgoth{o} = \mathbb{Z}[\omega]$ denote the ring of Eisenstein integers, which is a  Euclidean domain and hence a PID (see \cite[Section 8.1]{DF04}). Then $\textgoth{o}^{\times} = \{ \pm 1, \pm \omega, \pm \omega^2 \}$. The representatives of the $\textgoth{o}^{\times}$-orbits in $\textgoth{o}/\textgoth{o}^{\times}$ are depicted pictorially by the sector $0 \leq \arg(z) < \pi/3$.

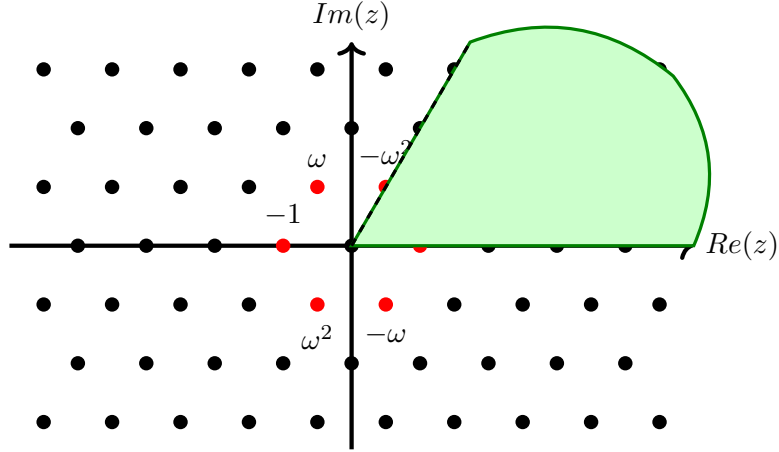
\begin{center}
\begin{tikzpicture}[scale = 0.9]
\draw[->,ultra thick] (-5,0)--(5,0) node[right]{$Re(z)$};
\draw[->,ultra thick] (0,-3)--(0,3) node[above]{$Im(z)$};
\fill (0, 0) circle[radius = 3pt];
\fill [red] (1, 0) circle[radius = 3pt];
\node[label={$1$}] at (1, 0) {};
\fill [red] (-1, 0) circle[radius = 3pt];
\node[label={$-1$}] at (-1, 0) {};
\fill [red] (1/2, {sqrt(3)/2}) circle[radius = 3pt];
\node[label={$-\omega^2$}] at (1/2, {sqrt(3)/2}) {};
\fill [red] (-1/2, {sqrt(3)/2}) circle[radius = 3pt];
\node[label={$\omega$}] at (-1/2, {sqrt(3)/2}) {};
\fill [red] (1/2, {-sqrt(3)/2}) circle[radius = 3pt];
\node[label={270:$- \omega$}] at (1/2, {-sqrt(3)/2}) {};
\fill [red] (-1/2, {-sqrt(3)/2}) circle[radius = 3pt];
\node[label={270:$\omega^2$}] at (-1/2, {-sqrt(3)/2}) {};
\fill (2, 0) circle[radius = 3pt];
\fill (3, 0) circle[radius = 3pt];
\fill (4, 0) circle[radius = 3pt];
\fill (-2, 0) circle[radius = 3pt];
\fill (-3, 0) circle[radius = 3pt];
\fill (-4, 0) circle[radius = 3pt];
\fill (3/2, {sqrt(3)/2}) circle[radius = 3pt];
\fill (5/2, {sqrt(3)/2}) circle[radius = 3pt];
\fill (7/2, {sqrt(3)/2}) circle[radius = 3pt];
\fill (9/2, {sqrt(3)/2}) circle[radius = 3pt];
\fill (-3/2, {sqrt(3)/2}) circle[radius = 3pt];
\fill (-5/2, {sqrt(3)/2}) circle[radius = 3pt];
\fill (-7/2, {sqrt(3)/2}) circle[radius = 3pt];
\fill (-9/2, {sqrt(3)/2}) circle[radius = 3pt];
\fill (3/2, {-sqrt(3)/2}) circle[radius = 3pt];
\fill (5/2, {-sqrt(3)/2}) circle[radius = 3pt];
\fill (7/2, {-sqrt(3)/2}) circle[radius = 3pt];
\fill (9/2, {-sqrt(3)/2}) circle[radius = 3pt];
\fill (-3/2, {-sqrt(3)/2}) circle[radius = 3pt];
\fill (-5/2, {-sqrt(3)/2}) circle[radius = 3pt];
\fill (-7/2, {-sqrt(3)/2}) circle[radius = 3pt];
\fill (-9/2, {-sqrt(3)/2}) circle[radius = 3pt];
\fill (4, {sqrt(3)}) circle[radius = 3pt];
\fill (3, {sqrt(3)}) circle[radius = 3pt];
\fill (2, {sqrt(3)}) circle[radius = 3pt];
\fill (1, {sqrt(3)}) circle[radius = 3pt];
\fill (0, {sqrt(3)}) circle[radius = 3pt];
\fill (-1, {sqrt(3)}) circle[radius = 3pt];
\fill (-2, {sqrt(3)}) circle[radius = 3pt];
\fill (-3, {sqrt(3)}) circle[radius = 3pt];
\fill (-4, {sqrt(3)}) circle[radius = 3pt];
\fill (4, {-sqrt(3)}) circle[radius = 3pt];
\fill (3, {-sqrt(3)}) circle[radius = 3pt];
\fill (2, {-sqrt(3)}) circle[radius = 3pt];
\fill (1, {-sqrt(3)}) circle[radius = 3pt];
\fill (0, {-sqrt(3)}) circle[radius = 3pt];
\fill (-1, {-sqrt(3)}) circle[radius = 3pt];
\fill (-2, {-sqrt(3)}) circle[radius = 3pt];
\fill (-3, {-sqrt(3)}) circle[radius = 3pt];
\fill (-4, {-sqrt(3)}) circle[radius = 3pt];
\fill (1/2, {sqrt(27)/2}) circle[radius = 3pt];
\fill (3/2, {sqrt(27)/2}) circle[radius = 3pt];
\fill (5/2, {sqrt(27)/2}) circle[radius = 3pt];
\fill (7/2, {sqrt(27)/2}) circle[radius = 3pt];
\fill (9/2, {sqrt(27)/2}) circle[radius = 3pt];
\fill (-1/2, {sqrt(27)/2}) circle[radius = 3pt];
\fill (-3/2, {sqrt(27)/2}) circle[radius = 3pt];
\fill (-5/2, {sqrt(27)/2}) circle[radius = 3pt];
\fill (-7/2, {sqrt(27)/2}) circle[radius = 3pt];
\fill (-9/2, {sqrt(27)/2}) circle[radius = 3pt];
\fill (1/2, {-sqrt(27)/2}) circle[radius = 3pt];
\fill (3/2, {-sqrt(27)/2}) circle[radius = 3pt];
\fill (5/2, {-sqrt(27)/2}) circle[radius = 3pt];
\fill (7/2, {-sqrt(27)/2}) circle[radius = 3pt];
\fill (9/2, {-sqrt(27)/2}) circle[radius = 3pt];
\fill (-1/2, {-sqrt(27)/2}) circle[radius = 3pt];
\fill (-3/2, {-sqrt(27)/2}) circle[radius = 3pt];
\fill (-5/2, {-sqrt(27)/2}) circle[radius = 3pt];
\fill (-7/2, {-sqrt(27)/2}) circle[radius = 3pt];
\fill (-9/2, {-sqrt(27)/2}) circle[radius = 3pt];
\filldraw[fill=green!20!white,
draw=green!50!black, opacity = 0.3, very thick] (0,0) -- ({3/sqrt(3)}, 3) to[bend left] (4.7, 2.5) to[bend left] (5, 0) -- cycle;
\draw [draw = black, very thick, dashed] (0,0) -- ({3/sqrt(3)}, 3);
\end{tikzpicture}
\end{center}
Each black point corresponds to an element of $\textgoth{o}$ and the six red points are the units of $\textgoth{o}$. The shaded green sector without the dashed black line contains the elements of $\textgoth{o}/\textgoth{o}^{\times}$ which index the ideals in $\textgoth{o}$. For example, the point $- \omega^2 = 1 + \omega$ lies on the dashed black line, but $- \omega^2 \textgoth{o} = \textgoth{o}$ as ideals in $\textgoth{o}$. 
\end{example}

Returning to the general case where $R$ is an arbitrary PID, suppose that $(a, b)^T \in M_{2 \times 1}(R)$. Steinberg reduction produces a specific matrix

\begin{equation} \label{equation: Steinberg Goal}
\begin{pmatrix}
p & q \\
r & s \\
\end{pmatrix} \in Y(R) \qquad \text{such that} \qquad
\begin{pmatrix}
p & q \\
r & s \\
\end{pmatrix}
\begin{pmatrix}
a \\
b \\
\end{pmatrix} = 
\begin{pmatrix}
\gcd(a, b) \\
0 \\
\end{pmatrix}.
\end{equation}

The set $Y(R)$ in equation \eqref{equation: YR} was defined by Steinberg to produce the specific matrices required for the variant of Bruhat decomposition in \cite[Theorem 15]{Ste67}. \newline

First assume that $a \neq 0$ and $b \neq 0$. Select $r \in (R - \{ 0 \})/R^{\times}$ and $s \in R$ such that $a = \gcd(a, b) s$ and $b = - \gcd(a, b) r$. Note that since $r \in (R - \{ 0 \})/R^{\times}$, $r$ and $s$ are both unique elements of $R$ satisfying the equations

\begin{equation*}
ra + sb = 0 \qquad \text{and} \qquad \gcd(r, s) = 1. 
\end{equation*}

Since $\gcd(r, s) = 1$, choose $p, q \in R$ such that $ps - qr = 1$. Then, there exists a unique $v \in R$ such that $p_v = p- rv$ and $q_v = q - sv$ satisfy

\begin{equation*}
p_v \in R/rR \qquad \text{and} \qquad p_v s - q_v r = 1.
\end{equation*} 

Therefore, the specific matrix

\begin{equation*}
\begin{pmatrix}
p_v & q_v \\
r & s \\
\end{pmatrix} \in Y(R) \qquad \text{satisfies} \qquad
\begin{pmatrix}
p_v & q_v \\
r & s \\
\end{pmatrix}
\begin{pmatrix}
a \\
b \\
\end{pmatrix} = 
\begin{pmatrix}
\gcd(a, b) \\
0 \\
\end{pmatrix}
\end{equation*}

because $p_v a + q_v b = (p_v s - q_v r) \gcd(a, b) = \gcd(a, b)$. \newline

Secondly, if $a = 0$ and $b \neq 0$ then the matrix

\begin{equation*}
\begin{pmatrix}
0 & -1 \\
1 & 0 \\
\end{pmatrix} \in Y(R) \qquad \text{satisfies} \qquad
\begin{pmatrix}
0 & -1 \\
1 & 0 \\
\end{pmatrix}
\begin{pmatrix}
0 \\
b \\
\end{pmatrix} = 
\begin{pmatrix}
-b \\
0 \\
\end{pmatrix}.
\end{equation*}

Finally, if $b = 0$ then the matrix

\begin{equation*}
\begin{pmatrix}
1 & 0 \\
0 & 1 \\
\end{pmatrix} \in Y(R) \qquad \text{satisfies} \qquad
\begin{pmatrix}
1 & 0 \\
0 & 1 \\
\end{pmatrix}
\begin{pmatrix}
a \\
0 \\
\end{pmatrix} = 
\begin{pmatrix}
a \\
0 \\
\end{pmatrix}.
\end{equation*} 

\section{The orbit space $\Gamma_{\infty}(3) \backslash \Gamma(3)$}

\subsection{Definition and properties of $\Lambda^1$ and $\Lambda^2$ invariants} 

In this section, we define the matrix orbit space $\Gamma_{\infty}(3)\backslash \Gamma(3)$ and the favourite invariants of the constituent orbits. A useful way of understanding matrix orbits is to look for \textbf{invariants} --- elements which remain the same when $A$ is multiplied on the left by an element of the group $\Gamma_{\infty}(3)$. 

\begin{definition} \label{definition: Gamma}
Let $\omega = e^{2 \pi i/3}$ and $\mathbb{Z}[\omega] = \textgoth{o}$ be the ring of Eisenstein integers. Define

\begin{equation*}
\begin{matrix}
\Gamma(3) & = & \{ A \in SL_3(\textgoth{o}) \; \vert \; A \equiv I_3 \text{ mod } 3\textgoth{o} \}, \\
\\
\Gamma_{\infty}(3) & = & \Gamma(3) \cap U, \\
\\
\Gamma_{\infty}(3) \backslash \Gamma(3) & = & \{ \Gamma_{\infty}(3) \cdot A \; \vert \; A \in \Gamma(3) \}.
\end{matrix}
\end{equation*}
where $I_3$ is the $3 \times 3$ identity matrix, the congruence $A \equiv I_3$ mod $3\textgoth{o}$ is computed entry by entry and $U$ denotes the subgroup of upper triangular matrices in $SL_3(\textgoth{o})$. Equivalently, $\Gamma_{\infty}(3)$ is the subgroup of upper triangular unipotent matrices in $\Gamma(3)$. 
\end{definition}

The question we will focus on is: If $A \in \Gamma(3)$, then what is a matrix representative for the orbit $\Gamma_{\infty}(3) \cdot A$? The definition of invariants we give uses \emph{wedge product matrices}. 

\begin{definition}
Let $R$ be a commutative ring and

\begin{equation*}
A = 
\begin{pmatrix}
a & b & c \\
d & e & f \\
g & h & i 
\end{pmatrix} \in M_{3 \times 3}(R).
\end{equation*}
Define the \textbf{wedge product matrices} of $A$ by
\begin{equation} \label{equation: Second Wedge Product}
\Lambda^1(A) = A \qquad \text{and} \qquad
\Lambda^2(A) = 
\begin{pmatrix}
ae - bd & af - cd & bf - ec \\
ah - bg & ai - cg & bi - hc \\
dh - eg & di - fg & ei - fh
\end{pmatrix}.
\end{equation}

The matrix
\begin{equation} \label{equation: Upsilon}
\Upsilon^1(A) = 
\begin{pmatrix}
i & -f & c \\
-h & e & -b \\
g & -d & a
\end{pmatrix}
\end{equation}

satisfies the relations 

\begin{equation} \label{equation: Adjugate}
\Lambda^2(A) \Upsilon^1(A) = \Upsilon^1(A) \Lambda^2(A) = \det(A) I_3
\end{equation}

and

\begin{equation} \label{equation: Upsilon Multiplication}
\Upsilon^1(AB) = \Upsilon^1(B) \Upsilon^1(A). 
\end{equation} 
\end{definition}

\noindent Observe that if $A \in M_{3 \times 3}(R)$ then the entries of the matrix $\Lambda^2(A)$ are all the possible $2 \times 2$ minors of $A$. The wedge product matrices are also  known as \emph{compound matrices} (see the reference \cite{Mul98} for instance). 

\begin{remark}
\emph{We call $\Upsilon^1(A)$ the first adjugate matrix of $A$ because if $B \in M_{n \times n}(R)$ then the classical notion of the adjugate matrix of $B$ is a matrix $adj(B) \in M_{n \times n}(R)$ satisfying (see \cite[Corollary 9.161]{Rot03})}

\begin{equation*}
B \; adj(B) = adj(B) \; B = \det(B) I_n
\end{equation*}
\emph{which is similar to equation \eqref{equation: Adjugate}.}
\end{remark}

Now we define the invariants of a matrix $A \in \Gamma(3)$. 

\begin{definition} \label{definition: Gamma Invariants}
The \textbf{$\Lambda^1$ and $\Lambda^2$ invariants} of 
\begin{equation*}
A = 
\begin{pmatrix}
a & b & c \\
d & e & f \\
g & h & i \\
\end{pmatrix} \in \Gamma(3)
\end{equation*}
are 
\begin{equation} \label{equation: Lambda 1 Invariants}
A_1 = g, \qquad B_1 = h, \qquad C_1 = i
\end{equation}
and
\begin{equation} \label{equation: Lambda 2 Invariants}
A_2 = dh - eg, \qquad B_2 = di-fg, \qquad C_2 = ei-fh
\end{equation}
respectively. We also define

\begin{equation*}
Inv(A) = (A_1, B_1, C_1, A_2, B_2, C_2) \in \textgoth{o}^6.
\end{equation*}
\end{definition}
\noindent The $\Lambda^1$ invariants form the bottom row of $A$ and the $\Lambda^2$ invariants form the bottom row of $\Lambda^2(A)$ (see equation \eqref{equation: Second Wedge Product}). \newline

The $\Lambda^1$ and $\Lambda^2$ invariants of a matrix $A \in \Gamma(3)$ satisfy a specific set of conditions which we define below. 

\begin{definition}
The \textbf{invariant conditions} on $(A_1, B_1, C_1, A_2, B_2, C_2) \in \textgoth{o}^6$ are

\begin{equation} \label{equation: I1}
A_1 \equiv A_2 \equiv B_1 \equiv B_2 \equiv 0 \text{ mod } 3 \textgoth{o}, \tag{I1}
\end{equation}

\begin{equation} \label{equation: I2}
C_1 \equiv C_2 \equiv 1 \text{ mod } 3 \textgoth{o}, \tag{I2}
\end{equation}

\begin{equation} \label{equation: I3}
\gcd(A_1, B_1, C_1) = \gcd(A_2, B_2, C_2) = 1, \tag{I3}
\end{equation}

\begin{equation} \label{equation: I4}
A_1 C_2 - B_1 B_2 + C_1 A_2 = 0. \tag{I4}
\end{equation}
\end{definition}

\begin{proposition} \label{proposition: Invariant Properties}
Let $A \in \Gamma(3)$. Then $Inv(A) \in \textgoth{o}^6$ satisfies the invariant conditions \eqref{equation: I1}, \eqref{equation: I2}, \eqref{equation: I3} and \eqref{equation: I4}. 
\end{proposition}

\begin{proof}
Assume that 

\begin{equation*}
A = 
\begin{pmatrix}
a & b & c \\
d & e & f \\
g & h & i \\
\end{pmatrix} \in \Gamma(3) \qquad \text{and} \qquad Inv(A) = (A_1, B_1, C_1, A_2, B_2, C_2). 
\end{equation*} 
By using the definition of $\Gamma(3)$ in Definition \ref{definition: Gamma} and the definition of $\Lambda^1$ and $\Lambda^2$ invariants in equations \eqref{equation: Lambda 1 Invariants} and \eqref{equation: Lambda 2 Invariants}, a direct computation yields the conditions \eqref{equation: I1} and \eqref{equation: I2}. \newline
\phantom{} \newline
Since $A \in SL_3(\textgoth{o})$, then $\det(A) = 1$ and Laplace expansion (see \cite[Proposition 9.160]{Rot03}) along the bottom row of $A$ yields

\begin{equation*}
1 = g(bf - ec) - h(af - cd) + i(ae - bd) \in g \textgoth{o} + h \textgoth{o} + i \textgoth{o}.
\end{equation*}
Hence, $\gcd(A_1, B_1, C_1) = \gcd(g, h, i) = 1$. Similarly, Laplace expansion along the top row of $A$ yields

\begin{equation*}
1 = a(ei - fh) - b(di - fg) + c(dh - eg) \in C_2 \textgoth{o} + B_2 \textgoth{o} + A_2 \textgoth{o}.
\end{equation*}
So, $\gcd(A_2, B_2, C_2) = 1$. Equation \eqref{equation: I4} follows from the direct computation

\begin{equation*}
A_1 C_2 - B_1 B_2 + C_1 A_2 = g(ei - fh) - h(di - fg) + i(dh-eg) = 0. \qedhere
\end{equation*}
\end{proof}

The next theorem establishes a bijection between sets of invariants satisfying the invariant conditions and orbits in $\Gamma_{\infty}(3) \backslash \Gamma(3)$.  

\begin{theorem} \label{theorem: Gamma 3 Invariants}
Let $A, B \in \Gamma(3)$. Then, 

\begin{equation*}
\Gamma_{\infty}(3) \cdot A = \Gamma_{\infty}(3) \cdot B \qquad \text{if and only if} \qquad Inv(A) = Inv(B).
\end{equation*}
\end{theorem}

\begin{proof}
Assume that $A, B \in \Gamma(3)$. \newline
\phantom{} \newline
$\implies$: Suppose that $\Gamma_{\infty}(3) \cdot A = \Gamma_{\infty}(3) \cdot B$. Then there exists $C \in \Gamma_{\infty}(3)$ such that $CA = B$. Since the bottom rows of $C$ and $\Lambda^2(C)$ are both $[0, 0, 1]$, a quick computation of the bottom rows of $CA$ and $\Lambda^2(CA)$ gives $Inv(A) = Inv(CA) = Inv(B)$. \newline
\phantom{} \newline
$\impliedby$: Now assume that $Inv(A) = Inv(B) = (A_1, B_1, C_1, A_2, B_2, C_2)$. Since $A \in \Gamma(3)$ is invertible, let $D = BA^{-1}$. If

\begin{equation*}
A = 
\begin{pmatrix}
a_1 & b_1 & c_1 \\
d_1 & e_1 & f_1 \\
A_1 & B_1 & C_1 \\
\end{pmatrix}
\qquad \text{and} \qquad
B = 
\begin{pmatrix}
a_2 & b_2 & c_2 \\
d_2 & e_2 & f_2 \\
A_1 & B_1 & C_1 \\
\end{pmatrix}
\end{equation*}
then

\begin{equation} \label{equation: BA-1}
D = B A^{-1} =
\begin{pmatrix}
a_2 & b_2 & c_2 \\
d_2 & e_2 & f_2 \\
A_1 & B_1 & C_1 \\
\end{pmatrix}
\begin{pmatrix}
C_2 & c_1 B_1 - b_1 C_1 & b_1 f_1 - e_1 c_1 \\
- B_2 & a_1 C_1 - c_1 A_1 & c_1 d_1 - a_1 f_1 \\
A_2 & b_1 A_1 - a_1 B_1 & a_1 e_1 - b_1 d_1 \\
\end{pmatrix}.
\end{equation}

To see that $D \in \Gamma_{\infty}(3)$, we must show that $D$ is upper triangular and unipotent. The entries $d_{ij}$ of $D$ are $d_{31} = A_1 C_2 - B_1 B_2 + A_2 C_1 = 0$ (by invariant condition \eqref{equation: I4}), 

\begin{equation*}
d_{32} = A_1(c_1 B_1 - b_1 C_1) + B_1(a_1 C_1 - c_1 A_1) + C_1(b_1 A_1 - a_1 B_1) = 0,
\end{equation*}
\begin{equation*}
d_{33} = 
\begin{vmatrix}
a_1 & b_1 & c_1 \\
d_1 & e_1 & f_1 \\
A_1 & B_1 & C_1 \\
\end{vmatrix} = \det(A) = 1,
\end{equation*} 
\begin{equation*}
d_{21} = 
\begin{vmatrix}
d_2 & e_2 & f_2 \\
d_2 & e_2 & f_2 \\
A_1 & B_1 & C_1 \\
\end{vmatrix} = 0,
\end{equation*} 
\begin{equation*}
d_{22} = 
\begin{vmatrix}
a_1 & b_1 & c_1 \\
d_2 & e_2 & f_2 \\
A_1 & B_1 & C_1 \\
\end{vmatrix} =
\begin{vmatrix}
a_1 & b_1 & c_1 \\
d_1 & e_1 & f_1 \\
A_1 & B_1 & C_1 \\
\end{vmatrix} = 
 1
\end{equation*}
since $Inv(A) = Inv(B)$ and finally, $d_{11} = \det(B) = 1$. So, $D = BA^{-1}$ is an upper triangular unipotent matrix in $\Gamma(3)$. Hence, $D \in \Gamma_{\infty}(3)$.
\end{proof}

\begin{remark}
\emph{In \cite[Page 484]{BH86}, Bump and Hoffstein define the involution $\iota: GL_3(\mathbb{C}) \rightarrow GL_3(\mathbb{C})$ by}

\begin{equation*}
^{\iota}A = 
\begin{pmatrix}
& & 1 \\
& 1 & \\
1 & & \\
\end{pmatrix}
(A^{-1})^T
\begin{pmatrix}
& & 1 \\
& 1 & \\
1 & & \\
\end{pmatrix}.
\end{equation*}
\emph{In \cite[Page 485]{BH86}, the invariants of $A \in \Gamma(3)$ are the elements of $\textgoth{o}$ which comprise the bottom rows of $A$ and $^{\iota}A$, which are denoted by $[A_1, B_1, C_1]$ and $[A_2, B_2, C_2]$ respectively. To see how this is related to Definition \ref{definition: Gamma Invariants}, we compute directly that}

\begin{equation} \label{equation: 1.4}
\Lambda^2(A) = \begin{pmatrix}
1 & & \\
& -1 & \\
& & 1 \\
\end{pmatrix} 
(^{\iota}A)
\begin{pmatrix}
1 & & \\
& -1 & \\
& & 1 \\
\end{pmatrix}.
\end{equation}
\emph{Equation \eqref{equation: 1.4} provides the crucial link between the bottom row $[A_2, B_2, C_2]$ of $^{\iota}A$ in \cite[p.~486]{BH86} and $Inv(A)$. In particular, the only difference between $A_2, B_2, C_2$ as in \cite[p.~486]{BH86} and the $\Lambda^2$ invariants in Definition \ref{definition: Gamma Invariants} is the sign of $B_2$. Thus, we have connected our approach to the invariants of $\Gamma_{\infty}(3) \backslash \Gamma(3)$ with that of Bump and Hoffstein.}
\end{remark}

\subsection{Decomposing an element of $SL_3$ with Steinberg reduction}

In this section, we will use Steinberg reduction to construct a particular decomposition of each element in $SL_3(\textgoth{o})$. In the same spirit as the Bruhat decomposition, we will also show that this decomposition is unique by decomposing $SL_3(\textgoth{o})$ into a disjoint union of three subsets and then proving uniqueness for each disjoint subset of $SL_3(\textgoth{o})$. Later, we will use the decomposition in this section to produce a matrix representative of an orbit in $\Gamma_{\infty}(3) \backslash \Gamma(3)$ with given $\Lambda^1$ and $\Lambda^2$ invariants. Let

\begin{equation} \label{equation: D3}
D(3) = \left\{
\begin{pmatrix}
i & & \\
& j & \\
& & k \\
\end{pmatrix} \in M_{3 \times 3}(\textgoth{o}) \; \vert \; ijk = 1 \right\} \subseteq SL_3(\textgoth{o}),
\end{equation}

\begin{equation} \label{equation: U3}
U(3) = \left\{
\begin{pmatrix}
1 & \alpha & \beta \\
& 1 & \gamma \\
& & 1 \\
\end{pmatrix}  \; \Big\vert \; \alpha, \beta, \gamma \in \{ 0, 1, 2 \} + \{ 0, 1, 2 \} \omega \right\} \subseteq SL_3(\textgoth{o})
\end{equation}
and let $\varphi_1, \varphi_2: SL_2(\textgoth{o}) \rightarrow SL_3(\textgoth{o})$ denote the group homomorphisms given by

\begin{equation*}
\varphi_1
\begin{pmatrix}
a & b \\
c & d \\
\end{pmatrix} = 
\begin{pmatrix}
a & b & \\
c & d &  \\
&  & 1 \\
\end{pmatrix} \qquad \text{and} \qquad
\varphi_2
\begin{pmatrix}
a & b \\
c & d \\
\end{pmatrix} = 
\begin{pmatrix}
1 & & \\
& a & b \\
& c & d \\
\end{pmatrix}.
\end{equation*}
Next we will define the disjoint subsets of $SL_3(\textgoth{o})$ required for the proof of uniqueness in the next section.

\begin{definition}
Define

\begin{equation} \label{equation: Delta 1}
\Delta_1 = \{ A = (a_{ij}) \in SL_3(\textgoth{o}) \; \vert \; \text{$a_{21} \neq 0$ or $a_{31} \neq 0$} \}, 
\end{equation}

\begin{equation} \label{equation: Delta 2}
\Delta_2 = \{ A = (a_{ij}) \in SL_3(\textgoth{o}) \; \vert \; \text{$a_{21} = a_{31} = 0$} \}, 
\end{equation}

\begin{equation} \label{equation: Delta 1 0}
\Delta_{1, 0} = \{ A = (a_{ij}) \in \Delta_1 \; \vert \; \text{$(\Lambda^2(A))_{3, 1} = a_{21} a_{32} - a_{22} a_{31} = 0$} \}
\end{equation}
and
\begin{equation} \label{equation: Delta 1 1}
\Delta_{1, 1} = \{ A = (a_{ij}) \in SL_3(\textgoth{o}) \; \vert \; \text{$(\Lambda^2(A))_{3, 1} =  a_{21} a_{32} - a_{22} a_{31} \neq 0$} \}.
\end{equation}
\end{definition}

By definition, one can verify that $\Delta_{1, 1} \subseteq \Delta_1$, $\Delta_1 = \Delta_{1, 0} \sqcup \Delta_{1, 1}$ and

\begin{equation} \label{equation: SL3 Disjoint Decomp}
SL_3(\textgoth{o}) = \Delta_1 \sqcup \Delta_2 = \Delta_{1, 0} \sqcup \Delta_{1, 1} \sqcup \Delta_2. 
\end{equation}
Now we will use Steinberg reduction to construct a specific decomposition of an element in $SL_3(\textgoth{o})$. 

\begin{theorem} \label{theorem: SL3 Existence}
Let $A \in SL_3(\textgoth{o})$. Then there exist $y_1, y_2, y_3 \in Y(\textgoth{o})$, $d \in D(3)$, $u \in U(3)$ and $C \in \Gamma_{\infty}(3)$ such that 

\begin{equation*} 
A = \varphi_2(y_1^{-1}) \varphi_1(y_2^{-1}) \varphi_2(y_3^{-1}) d u C.
\end{equation*}
Here, $Y(\textgoth{o})$ is the set in equation \eqref{equation: YR}. 
\end{theorem}

\begin{proof}
Assume that 

\begin{equation*}
A = 
\begin{pmatrix}
a & b & c \\
d & e & f \\
g & h & i \\
\end{pmatrix} \in SL_3(\textgoth{o}). 
\end{equation*} 
\textbf{Step 1:} (First column of $A$) Use Steinberg reduction on the first column to construct specific matrices $y_1, y_2 \in Y(\textgoth{o})$ which satisfy  

\begin{equation*}
y_1 
\begin{pmatrix}
d \\
g
\end{pmatrix} = 
\begin{pmatrix}
\gcd(d, g) \\
0
\end{pmatrix} \qquad \text{and} \qquad 
y_2
\begin{pmatrix}
a \\
\gcd(d, g)
\end{pmatrix} = 
\begin{pmatrix}
p \\
0
\end{pmatrix}.
\end{equation*}
Here, $p = \gcd(a, d, g) \in \textgoth{o}^{\times}$. Consequently, 

\begin{equation*}
\varphi_1(y_2) \varphi_2(y_1) A = 
\begin{pmatrix}
p & b' & c' \\
0 & e'  & f' \\
0 & h' & i'
\end{pmatrix}.
\end{equation*}
\textbf{Step 2:} (Second column of $A$) Use Steinberg reduction on the second column of $\varphi_1(y_2) \varphi_2(y_1) A$ to construct a specific matrix $y_3 \in Y(\textgoth{o})$ such that

\begin{equation*}
y_3 
\begin{pmatrix}
e' \\
h'
\end{pmatrix} = 
\begin{pmatrix}
q \\
0
\end{pmatrix} \qquad \text{where} \qquad q = \gcd(e', h') \in \textgoth{o}^{\times}. 
\end{equation*} 
Thus, 

\begin{equation*}
\varphi_2(y_3) \varphi_1(y_2) \varphi_2(y_1) A = 
\begin{pmatrix}
p & x & y \\
0 & q & z \\
0 & 0 & r
\end{pmatrix}.
\end{equation*}
\textbf{Step 3:} (Decomposing the upper triangular matrix) On the RHS, we have

\begin{equation*}
\begin{pmatrix}
p & x & y \\
0 & q & z \\
0 & 0 & r
\end{pmatrix} = 
\begin{pmatrix}
p & & \\
& q & \\
& & r \\
\end{pmatrix}
\begin{pmatrix}
1 & qrx & qr y \\
& 1 & pr z \\
& & 1 \\
\end{pmatrix}.
\end{equation*}
because $pqr = 1$. If $a + b \omega \in \textgoth{o}$ then define the map

\begin{equation} \label{equation: Mod 3 Map}
\begin{matrix}
(-)_3: & \textgoth{o} & \rightarrow & \{ 0, 1, 2 \} + \{0, 1, 2 \} \omega \\
& a + b \omega & \mapsto & (a + b \omega)_3 = d + e \omega
\end{matrix}
\end{equation}
where $d, e \in \{ 0, 1, 2 \}$ satisfy the congruence relations $d \equiv a$ mod 3 and $e \equiv b$ mod 3. We have

\begin{equation*}
\begin{pmatrix}
1 & qrx & qr y \\
& 1 & pr z \\
& & 1 \\
\end{pmatrix} = 
\begin{pmatrix}
1 & \alpha & \beta \\
& 1 & \gamma \\
& & 1 \\
\end{pmatrix}
\begin{pmatrix}
1 & qrx - \alpha & qr y - \alpha(prz - \gamma) - \beta \\
& 1 & prz - \gamma \\
& & 1 \\
\end{pmatrix}
\end{equation*}
where 

\begin{equation} \label{equation: Mod 3 Condition}
\alpha = (qrx)_3, \qquad \beta = (qr y - \alpha(prz - \gamma))_3 \qquad \text{and} \qquad \gamma = (prz)_3.
\end{equation}
Equation \eqref{equation: Mod 3 Condition} ensures that the matrix

\begin{equation*}
C = 
\begin{pmatrix}
1 & qrx - \alpha & qr y - \alpha(prz - \gamma) - \beta \\
& 1 & prz - \gamma \\
& & 1 \\
\end{pmatrix} \in \Gamma_{\infty}(3). 
\end{equation*}
and

\begin{equation*}
\begin{pmatrix}
1 & \alpha & \beta \\
& 1 & \gamma \\
& & 1 \\
\end{pmatrix} \in U(3).
\end{equation*}
Putting all of the computations together, we obtain the decomposition

\begin{equation} \label{equation: SL3 Existence}
A = \varphi_2(y_1^{-1}) \varphi_1(y_2^{-1}) \varphi_2(y_3^{-1}) 
\begin{pmatrix}
p & & \\
& q & \\
& & r \\
\end{pmatrix}
\begin{pmatrix}
1 & \alpha & \beta \\
& 1 & \gamma \\
& & 1 \\
\end{pmatrix}
C.
\end{equation}
\end{proof}

The most important special cases of Theorem \ref{theorem: SL3 Existence} are the decompositions of elements in $\Delta_{1, 0}$ and $\Delta_2$. If $A \in \Delta_{1, 0}$ then by equation \eqref{equation: Delta 1 0}, either

\begin{equation*}
A = 
\begin{pmatrix}
a & b & c \\
d & e & f \\
0 & 0 & i
\end{pmatrix} \qquad \text{or} \qquad 
A = 
\begin{pmatrix}
a & b & c \\
0 & 0 & f \\
g & h & i
\end{pmatrix}
\end{equation*}
where $d, g, h \in \textgoth{o} - \{ 0 \}$. By applying Theorem \ref{theorem: SL3 Existence} to $A$, we obtain the decomposition

\begin{equation} \label{equation: Existence Delta 1 0}
A = \varphi_2(y_1^{-1}) \varphi_1(y_2^{-1}) \varphi_2(I_2) d u C = \varphi_2(y_1^{-1}) \varphi_1(y_2^{-1}) d u C
\end{equation}
where $y_1, y_2 \in Y(\textgoth{o})$, $d \in D(3)$, $u \in U(3)$, $C \in \Gamma_{\infty}(3)$ and $I_2$ is the $2 \times 2$ identity matrix. If $A \in \Delta_2$ then by equation \eqref{equation: Delta 2},

\begin{equation*}
A = 
\begin{pmatrix}
a & b & c \\
0 & e & f \\
0 & h & i
\end{pmatrix}.
\end{equation*}
By Theorem \ref{theorem: SL3 Existence}, we obtain the decomposition

\begin{equation} \label{equation: Existence Delta 2}
A = \varphi_2(I_2) \varphi_1(I_2) \varphi_2(y^{-1}) d u C = \varphi_2(y^{-1}) d u C
\end{equation}
where $y \in Y(\textgoth{o})$, $d \in D(3)$, $u \in U(3)$ and $C \in \Gamma_{\infty}(3)$. 

\subsection{Uniqueness of the decomposition} 

By equation \eqref{equation: Delta 1 1}, whether a matrix $A \in SL_3(\textgoth{o})$ belongs to the set $\Delta_{1, 1}$ depends on the wedge product matrix $\Lambda^2(A)$, defined in equation \eqref{equation: Second Wedge Product}. Hence when dealing with $\Delta_{1, 1}$ in this section, we will make use of the following important properties of $\Lambda^2(A)$ which are proved by direct computation.

\begin{lemma} \label{lemma: Wedge 2}
Let $A, B \in SL_3(\textgoth{o})$ and $y \in M_{2 \times 2}(\textgoth{o})$. Then

\begin{enumerate}
\item $\Lambda^2(AB) = \Lambda^2(A) \Lambda^2(B)$, \\
\item $\Lambda^2(\varphi_1(y)) = \varphi_2(y)$, \\
\item $\Lambda^2(\varphi_2(y)) = \varphi_1(y)$. 
\end{enumerate}
\end{lemma} 

Lemma \ref{lemma: Matrix Not Identity} is useful for the proofs of uniqueness of the decompositions in equation \eqref{equation: Existence Delta 1 0} and equation \eqref{equation: SL3 Existence}.

\begin{lemma} \label{lemma: Matrix Not Identity}
Let $y_1, y_2, y_3 \in Y(\textgoth{o})$, $d \in D(3)$, $u \in U(3)$ and $C \in \Gamma_{\infty}(3)$. Let 

\begin{equation*}
A = \varphi_2(y_1^{-1}) \varphi_1(y_2^{-1}) \varphi_2(y_3^{-1}) d u C.
\end{equation*}

\begin{enumerate}
\item If $A \in \Delta_1$ then $y_2 \neq I_2$.
\item If $A \in \Delta_{1, 1}$ then $y_3 \neq I_2$. 
\end{enumerate}
\end{lemma}

\begin{proof}
Assume that $y_1, y_2, y_3 \in Y(\textgoth{o})$, $d \in D(3)$, $u \in U(3)$ and $C \in \Gamma_{\infty}(3)$. Let 

\begin{equation*}
A = \varphi_2(y_1^{-1}) \varphi_1(y_2^{-1}) \varphi_2(y_3^{-1}) d u C.
\end{equation*}
If $y_2 = I_2$ then by direct computation

\begin{equation*}
A = \varphi_2(y_1^{-1}) \varphi_2(y_3^{-1}) d u C \in \Delta_2. 
\end{equation*}
Since $SL_3(\textgoth{o}) = \Delta_1 \sqcup \Delta_2$ then $A \not \in \Delta_1$. Next assume that $y_3 = I_2$. By Lemma \ref{lemma: Wedge 2} and another direct computation, the $3, 1$ entry of the matrix

\begin{equation*}
\Lambda^2(A) = \Lambda^2(\varphi_2(y_1^{-1}) \varphi_1(y_2^{-1}) d u C) = \varphi_1(y_1^{-1}) \varphi_2(y_2^{-1}) \Lambda^2(duC)
\end{equation*}
is equal to 0. By equation \eqref{equation: Delta 1 1}, we must have $A \not \in \Delta_{1, 1}$. 
\end{proof}

We will repeatedly use the next lemma in our proofs of uniqueness. 

\begin{lemma} \label{lemma: Y Equality}
Let

\begin{equation*}
A = 
\begin{pmatrix}
p & q \\
r & s 
\end{pmatrix} \qquad \text{and} \qquad B = 
\begin{pmatrix}
p' & q' \\
r' & s'
\end{pmatrix}
\end{equation*} 
be elements of $Y(\textgoth{o}) - \{ I_2 \}$. If $rs' - r' s = 0$ then $A = B$. 
\end{lemma}

\begin{proof}
Assume that $A, B \in Y(\textgoth{o}) - \{ I_2 \}$ are the matrices defined in the statement of the lemma. Assume that $r s' - r' s = 0$. By definition of $Y(\textgoth{o})$ in equation \eqref{equation: YR}, $r, r' \in (\textgoth{o} - \{ 0 \})/\textgoth{o}^{\times}$. Hence

\begin{equation*}
\frac{s'}{r'} = \frac{s}{r}. 
\end{equation*}
Now $\det(A) = \det(B) = 1$ and subsequently, $\gcd(r, s) = \gcd(r', s') = 1$. So, the fractions in the above equation are fully simplified. In tandem with the fact that $r, r' \in (\textgoth{o} - \{ 0 \})/\textgoth{o}^{\times}$, we find that $r = r'$ and consequently $s = s'$. Arguing as in Steinberg reduction, there exists a unique pair of elements $x, y \in \textgoth{o}$ such that 

\begin{equation*}
x \in \textgoth{o}/r \textgoth{o} = \textgoth{o}/r' \textgoth{o} \qquad \text{and} \qquad x s - y r = x s' - y r' =  1. 
\end{equation*}
Observe that the pairs of elements $(p, q)$ and $(p', q')$ also satisfy the above conditions. By uniqueness, $p = x = p'$ and $q = y = q'$. Hence $A = B$ as required. 
\end{proof}

Below, we prove the uniqueness of the decompsitions of elements in $\Delta_{1, 0}, \Delta_{1, 1}$ and $\Delta_2$. The following uniqueness theorems are stated as cell decompositions in order to emphasise the connection between Steinberg reduction, used to produce the decompositions in equations \eqref{equation: SL3 Existence}, \eqref{equation: Existence Delta 1 0} and \eqref{equation: Existence Delta 2}, and the Bruhat decomposition in \cite[Theorem 15]{Ste67}. We begin with the $\Delta_2$ case. 

\begin{theorem} \label{theorem: Uniqueness Delta 2}
Let $\Delta_2$ be the subset of $SL_3(\textgoth{o})$ defined in equation \eqref{equation: Delta 2}. We have the following equality:

\begin{equation*}
\Delta_2 = \bigsqcup_{y \in Y(\textgoth{o})} \bigsqcup_{d \in D(3)} \bigsqcup_{u \in U(3)} \varphi_2(y^{-1}) du \Gamma_{\infty}(3). 
\end{equation*}
\end{theorem}

\begin{proof}
First assume that $y \in Y(\textgoth{o})$, $d \in D(3)$, $u \in U(3)$ and $C \in \Gamma_{\infty}(3)$. Let

\begin{equation*}
A = (a_{ij}) = \varphi_2(y^{-1}) d u C. 
\end{equation*}
By direct computation, $a_{21} = a_{31} = 0$ and so $A \in \Delta_2$. \newline

Conversely, assume that $A \in \Delta_2$. By equation \eqref{equation: Existence Delta 2}, there exist $y \in Y(\textgoth{o})$, $d \in D(3)$, $u \in U(3)$ and $C \in \Gamma_{\infty}(3)$ such that $A = \varphi_2(y^{-1}) d u C$. We will now show that this decomposition of $A$ is unique. To this end, assume that there exist $z \in Y(\textgoth{o})$, $d' \in D(3)$, $u' \in U(3)$ and $C' \in \Gamma_{\infty}(3)$ such that 

\begin{equation*}
\varphi_2(y^{-1}) d u C = A = \varphi_2(z^{-1}) d' u' C'. 
\end{equation*}
We will first show that $y = z$. Let
\begin{equation*}
y = 
\begin{pmatrix}
j & k \\
l & m
\end{pmatrix}, \qquad 
z = 
\begin{pmatrix}
j' & k' \\
l' & m' 
\end{pmatrix}, \qquad
d = 
\begin{pmatrix}
p & & \\
& q & \\
& & r
\end{pmatrix}, \qquad 
d' =
\begin{pmatrix}
p' & & \\
& q' & \\
& & r'
\end{pmatrix},
\end{equation*}
\begin{equation*}
u = 
\begin{pmatrix}
1 & \alpha & \beta \\
& 1 & \gamma \\
& & 1
\end{pmatrix} \qquad \text{and} \qquad
u' =
\begin{pmatrix}
1 & \alpha' & \beta' \\
& 1 & \gamma' \\
& & 1
\end{pmatrix}.
\end{equation*}
Then 

\begin{equation} \label{equation: Delta 2 Uniqueness}
\begin{pmatrix}
p & \ast & \ast \\
& qmj' - qlk' & \ast \\
& qml' - q m'l & \ast
\end{pmatrix} = \varphi_2(z) \varphi_2(y^{-1}) d u C = d'u' C' =
\begin{pmatrix}
p' & \ast & \ast \\
& q' & \ast \\
& & r'
\end{pmatrix}.
\end{equation} 
In particular, $q(ml' - m' l) = 0$ and since $q \in \textgoth{o}^{\times}$ then $ml' - m' l = 0$. There are now three cases to consider: \newline

\textbf{Case 1:} $l, l' \neq 0$. \newline

If $l, l' \neq 0$ then by definition of $Y(\textgoth{o})$ in equation \eqref{equation: YR}, $y, z \in Y(\textgoth{o}) - \{ I_2 \}$. By Lemma \ref{lemma: Y Equality}, we have $y = z$. \newline

\textbf{Case 2:} $l = 0$. \newline

If $l = 0$ then since $y \in Y(\textgoth{o})$, $y = I_2$. By equation \eqref{equation: Delta 2 Uniqueness}, $0 = q(ml' - m' l) = q(1l' - m'(0)) = ql'$. So $l' = 0$ and since $z \in Y(\textgoth{o})$, $z = I_2 = y$. \newline

\textbf{Case 3:} $l' = 0$. \newline

If $l' = 0$ then since $z \in Y(\textgoth{o})$, $z = I_2$. By equation \eqref{equation: Delta 2 Uniqueness}, $0 = q(ml' - m' l) = q(m(0) - 1l) = -ql$. Hence $l = 0$ and since $y \in Y(\textgoth{o})$ then $y = I_2 = z$. Cases 1, 2 and 3 together demonstrate that $y = z$. \newline

By invertibility of $y$, we now have $duC = d' u' C'$. Comparing the diagonal entries on both sides of the equation, we deduce that $p = p'$, $q = q'$, $r = r'$ and $d = d'$. By invertibility of $d$, $u C = u' C'$. Reducing both sides of this equation modulo $3 \textgoth{o}$, we obtain the congruence

\begin{equation*}
\begin{pmatrix}
1 & \alpha & \beta \\
& 1 & \gamma \\
& & 1
\end{pmatrix} \equiv
\begin{pmatrix}
1 & \alpha' & \beta' \\
& 1 & \gamma' \\
& & 1
\end{pmatrix} \text{ mod } 3 \textgoth{o}. 
\end{equation*}
By definition of $U(3)$ in equation \eqref{equation: U3}, $\alpha, \beta, \gamma, \alpha', \beta', \gamma' \in \{ 0, 1, 2 \} + \{0, 1, 2 \} \omega$. So, the above congruence is actually an equality and $u = u'$. Finally by invertibility of $u$, $C = C'$. We conclude that the decomposition of $A \in \Delta_2$ in equation \eqref{equation: Existence Delta 2} is unique. Consequently

\begin{equation*}
\Delta_2 \subseteq \bigsqcup_{y \in Y(\textgoth{o})} \bigsqcup_{d \in D(3)} \bigsqcup_{u \in U(3)} \varphi_2(y^{-1}) du \Gamma_{\infty}(3)
\end{equation*}
as required. 
\end{proof}

The next case we deal with is $\Delta_{1, 0}$. 

\begin{theorem} \label{theorem: Uniqueness Delta 1 0}
Let $\Delta_{1, 0}$ be the subset of $SL_3(\textgoth{o})$ defined in equation \eqref{equation: Delta 1 0}. We have the following equality:

\begin{equation*}
\Delta_{1, 0} = \bigsqcup_{\substack{y_1, y_2 \in Y(\textgoth{o}), \\ y_2 \neq I_2}} \bigsqcup_{d \in D(3)} \bigsqcup_{u \in U(3)} \varphi_2(y_1^{-1}) \varphi_1(y_2^{-1}) du \Gamma_{\infty}(3). 
\end{equation*}
\end{theorem}

\begin{proof}
First assume that $y_1, y_2 \in Y(\textgoth{o})$, $y_2 \neq I_2$, $d \in D(3)$, $u \in U(3)$ and $C \in \Gamma_{\infty}(3)$. Let

\begin{equation*}
A = (a_{ij}) = \varphi_2(y_1^{-1}) \varphi_1(y_2^{-1}) du C.
\end{equation*}
By Lemma \ref{lemma: Matrix Not Identity}, $A \not \in \Delta_{1, 1}$. Now let

\begin{equation*}
y_1 = 
\begin{pmatrix}
j_1 & k_1 \\
l_1 & m_1
\end{pmatrix} \qquad \text{and} \qquad
y_2 = 
\begin{pmatrix}
j_2 & k_2 \\
l_2 & m_2
\end{pmatrix}. 
\end{equation*}
Since $y_2 \in Y(\textgoth{o}) - \{ I_2 \}$ then $l_2 \neq 0$. The computation divides into two cases. Firstly, if $y_1 = I_2$ then $m_1 = 1$ and by direct computation, $a_{21} \neq 0$. Secondly, if $y_1 \in Y(\textgoth{o}) - \{ I_2 \}$ then $l_1 \neq 0$ and $a_{31} \neq 0$. In either case, we find that $A \in \Delta_1 - \Delta_{1, 1} = \Delta_{1, 0}$ as required. \newline

Conversely, assume that $A \in \Delta_{1, 0}$. By equation \eqref{equation: Existence Delta 1 0}, there exist $y_1, y_2 \in Y(\textgoth{o})$, $d \in D(3)$, $u \in U(3)$ and $C \in \Gamma_{\infty}(3)$ such that 

\begin{equation*}
A = \varphi_2(y_1^{-1}) \varphi_1(y_2^{-1}) du C. 
\end{equation*}
By Lemma \ref{lemma: Matrix Not Identity}, $y_2 \neq I_2$. We will now show that this decomposition is unique. Assume that there exist $z_1, z_2 \in Y(\textgoth{o})$, $d' \in D(3)$, $u' \in U(3)$ and $C' \in \Gamma_{\infty}(3)$ such that $z_2 \neq I_2$ and

\begin{equation*}
\varphi_2(y_1^{-1}) \varphi_1(y_2^{-1}) du C = A = \varphi_2(z_1^{-1}) \varphi_1(z_2^{-1}) d'u' C'. 
\end{equation*}
If $i \in \{1, 2 \}$ then let

\begin{equation*}
y_i = 
\begin{pmatrix}
j_i & k_i \\
l_i & m_i
\end{pmatrix} \qquad \text{and} \qquad 
z_i = 
\begin{pmatrix}
j_i' & k_i' \\
l_i' & m_i'
\end{pmatrix}.
\end{equation*}
By assumption, we have

\begin{equation} \label{equation: Delta 1 0 Uniqueness 1}
\varphi_2(z_1) \varphi_2(y_1^{-1}) \varphi_1(y_2^{-1}) du C = \varphi_1(z_2^{-1}) d'u' C'.
\end{equation}
Computing the $3, 1$ entry for both sides of equation \eqref{equation: Delta 1 0 Uniqueness 1} gives $l_2(l_1 m_1' - l_1' m_1) = 0$. Since $y_2 \in Y(\textgoth{o}) - \{ I_2 \}$ then $l_2 \neq 0$ and $l_1 m_1' - l_1' m_1 = 0$. In order to show that $y_1 = z_1$, there are three cases to consider --- $l_1, l_1' \neq 0$, $l_1 = 0$ and $l_1' = 0$. If $l_1, l_1' \neq 0$ then Lemma \ref{lemma: Y Equality} gives $y_1 = z_1$. In the latter two cases, we obtain $y_1 = z_1 = I_2$ from equation \eqref{equation: Delta 1 0 Uniqueness 1} and the fact that $l_2 \neq 0$. \newline

So $y_1 = z_1$ and by invertibility of $y_1$, 

\begin{equation*}
\varphi_1(z_2) \varphi_1(y_2^{-1}) d uC = d' u' C'.
\end{equation*}
By comparing the $2, 1$ entry of both sides of the equation, we find that $l_2' m_2 - l_2 m_2' = 0$. Since $y_2, z_2 \in Y(\textgoth{o}) - \{ I_2 \}$ by assumption, we can apply Lemma \ref{lemma: Y Equality} to obtain $y_2 = z_2$. So by invertibility of $y_2$, $d u C = d' u' C'$. Arguing as in the proof of Theorem \ref{theorem: Uniqueness Delta 2}, we find that $d = d'$, $u = u'$ and $C = C'$. So, the decomposition of $A \in \Delta_{1, 0}$ in equation \eqref{equation: Existence Delta 1 0} is unique and

\begin{equation*}
\Delta_{1, 0} \subseteq \bigsqcup_{\substack{y_1, y_2 \in Y(\textgoth{o}), \\ y_2 \neq I_2}} \bigsqcup_{d \in D(3)} \bigsqcup_{u \in U(3)} \varphi_2(y_1^{-1}) \varphi_1(y_2^{-1}) du \Gamma_{\infty}(3).
\end{equation*}
This completes the proof. 
\end{proof}

Finally, we show uniqueness of the decomposition in equation \eqref{equation: SL3 Existence} for elements in $\Delta_{1, 1}$. 

\begin{theorem} \label{theorem: Uniqueness Delta 1 1}
Let $\Delta_{1, 1}$ be the subset of $SL_3(\textgoth{o})$ defined in equation \eqref{equation: Delta 1 1}. We have the following equality:

\begin{equation*}
\Delta_{1, 1} = \bigsqcup_{\substack{y_1, y_2, y_3 \in Y(\textgoth{o}), \\ y_2, y_3 \neq I_2}} \bigsqcup_{d \in D(3)} \bigsqcup_{u \in U(3)} \varphi_2(y_1^{-1}) \varphi_1(y_2^{-1}) \varphi_2(y_3^{-1}) du \Gamma_{\infty}(3). 
\end{equation*}
\end{theorem}

\begin{proof}
First assume that $y_1, y_2, y_3 \in Y(\textgoth{o})$, $y_2, y_3 \neq I_2$, $d \in D(3)$, $u \in U(3)$ and $C \in \Gamma_{\infty}(3)$. Let

\begin{equation*}
A = (a_{ij}) = \varphi_2(y_1^{-1}) \varphi_1(y_2^{-1}) \varphi_2(y_3^{-1}) du C.
\end{equation*}
By Lemma \ref{lemma: Wedge 2}, $\Lambda^2(A) = \varphi_1(y_1^{-1}) \varphi_2(y_2^{-1}) \varphi_1(y_3^{-1}) \Lambda^2(duC)$ and $(\Lambda^2(A))_{3, 1} = p l_2 l_3$ where $p \in \textgoth{o}^{\times}$. Since $y_2, y_3 \in Y(\textgoth{o}) - \{ I_2 \}$ then $l_2, l_3 \neq 0$ and $(\Lambda^2(A))_{3, 1} \neq 0$. Subsequently, $A \in \Delta_{1, 1}$. \newline

Conversely, assume that $A \in \Delta_{1, 1}$. By Theorem \ref{theorem: SL3 Existence}, there exist $y_1, y_2, y_3 \in \textgoth{o}$, $d \in D(3)$, $u \in U(3)$ and $C \in \Gamma_{\infty}(3)$ such that 

\begin{equation*}
A = \varphi_2(y_1^{-1}) \varphi_1(y_2^{-1}) \varphi_2(y_3^{-1}) du C.
\end{equation*}
By Lemma \ref{lemma: Matrix Not Identity}, $y_2, y_3 \in Y(\textgoth{o}) - \{ I_2 \}$. To see that the above  decomposition of $A$ is unique, assume that there exist $z_1, z_2, z_3 \in Y(\textgoth{o})$, $d' \in D(3)$, $u' \in U(3)$ and $C' \in \Gamma_{\infty}(3)$ such that $z_2, z_3 \in Y(\textgoth{o}) - \{ I_2 \}$ and

\begin{equation*}
\varphi_2(y_1^{-1}) \varphi_1(y_2^{-1}) \varphi_2(y_3^{-1}) du C = A = \varphi_2(z_1^{-1}) \varphi_1(z_2^{-1}) \varphi_2(z_3^{-1}) d'u' C'. 
\end{equation*}
If $i \in \{1, 2, 3 \}$ then let

\begin{equation*}
y_i = 
\begin{pmatrix}
j_i & k_i \\
l_i & m_i
\end{pmatrix} \qquad \text{and} \qquad 
z_i = 
\begin{pmatrix}
j_i' & k_i' \\
l_i' & m_i'
\end{pmatrix}.
\end{equation*}
Then we have

\begin{equation} \label{equation: Delta 1 1 Uniqueness 1}
\varphi_2(z_1) \varphi_2(y_1^{-1}) \varphi_1(y_2^{-1}) \varphi_2(y_3^{-1}) du C = \varphi_1(z_2^{-1}) \varphi_2(z_3^{-1}) d'u' C'
\end{equation}
and by computing the $3, 1$ entry, we obtain the equation $l_2(l_1 m_1' - l_1' m_1) = 0$. Since $y_2 \in Y(\textgoth{o}) - \{ I_2 \}$ then $l_2 \neq 0$ and $l_1 m_1' - l_1' m_1$. As in the proofs of Theorem \ref{theorem: Uniqueness Delta 2} and Theorem \ref{theorem: Uniqueness Delta 1 0}, there are three cases to consider --- $l_1, l_1' \neq 0$, $l_1 = 0$ and $l_1' = 0$. In the first case, we apply Lemma \ref{lemma: Y Equality} to obtain $y_1 = z_1$. In the latter two cases, we obtain $y_1 = I_2 = z_1$ from $l_2 \neq 0$ and equation \eqref{equation: Delta 1 1 Uniqueness 1}. \newline

So $y_1 = z_1$. By invertibility of $y_1$, we now have

\begin{equation*}
\varphi_1(y_2^{-1}) \varphi_2(y_3^{-1}) du C = \varphi_1(z_2^{-1}) \varphi_2(z_3^{-1}) d'u' C'.
\end{equation*}
By applying $\Lambda^2$ to both sides and then multiplying on the left by $\varphi_2(z_2)$, we obtain from Lemma \ref{lemma: Wedge 2}

\begin{equation*}
\varphi_2(z_2) \varphi_2(y_2^{-1}) \varphi_1(y_3^{-1}) \Lambda^2(du C) = \varphi_1(z_3^{-1}) \Lambda^2(d'u' C').
\end{equation*}
Again by computing the $3, 1$ entry of both sides, we obtain $l_3(l_2 m_2' - l_2' m_2) = 0$. Since $y_3 \in Y(\textgoth{o}) - \{ I_2 \}$ then $l_3 \neq 0$ and $l_2 m_2' - l_2' m_2 = 0$. By Lemma \ref{lemma: Y Equality}, $y_2 = z_2$ and

\begin{equation*}
\varphi_2(y_3^{-1}) du C = \varphi_2(z_3^{-1}) d'u' C'.
\end{equation*}
By repeating the argument in the proof of Theorem \ref{theorem: Uniqueness Delta 2}, we deduce that $y_3 = z_3$, $d = d'$, $u = u'$ and $C = C'$. Thus, the decomposition of $A \in \Delta_{1, 1}$ is unique and

\begin{equation*}
\Delta_{1, 1} \subseteq \bigsqcup_{\substack{y_1, y_2, y_3 \in Y(\textgoth{o}), \\ y_2, y_3 \neq I_2}} \bigsqcup_{d \in D(3)} \bigsqcup_{u \in U(3)} \varphi_2(y_1^{-1}) \varphi_1(y_2^{-1}) \varphi_2(y_3^{-1}) du \Gamma_{\infty}(3). 
\end{equation*}
This completes the proof. 
\end{proof}

By combining Theorem \ref{theorem: Uniqueness Delta 2}, Theorem \ref{theorem: Uniqueness Delta 1 0}, Theorem \ref{theorem: Uniqueness Delta 1 1} and equation \eqref{equation: SL3 Disjoint Decomp}, we obtain a ``Bruhat-like" decomposition of $SL_3(\textgoth{o})$.

\begin{corollary} \label{corollary: Bruhat Right}
The set of matrices $SL_3(\textgoth{o})$ is the disjoint union of the three cell decompositions

\begin{equation*}
\Delta_2 = \bigsqcup_{y \in Y(\textgoth{o})} \bigsqcup_{d \in D(3)} \bigsqcup_{u \in U(3)} \varphi_2(y^{-1}) du \Gamma_{\infty}(3),
\end{equation*}

\begin{equation*}
\Delta_{1, 0} = \bigsqcup_{\substack{y_1, y_2 \in Y(\textgoth{o}), \\ y_2 \neq I_2}} \bigsqcup_{d \in D(3)} \bigsqcup_{u \in U(3)} \varphi_2(y_1^{-1}) \varphi_1(y_2^{-1}) du \Gamma_{\infty}(3), 
\end{equation*}
and
\begin{equation*}
\Delta_{1, 1} = \bigsqcup_{\substack{y_1, y_2, y_3 \in Y(\textgoth{o}), \\ y_2, y_3 \neq I_2}} \bigsqcup_{d \in D(3)} \bigsqcup_{u \in U(3)} \varphi_2(y_1^{-1}) \varphi_1(y_2^{-1}) \varphi_2(y_3^{-1}) du \Gamma_{\infty}(3). 
\end{equation*}
\end{corollary}

\section{Constructing a representative of each orbit in $\Gamma_{\infty}(3) \backslash \Gamma(3)$}

\subsection{The form of a representative of each orbit in $\Gamma_{\infty}(3) \backslash \Gamma(3)$}

If $A \in \Gamma(3)$ then one can apply Corollary \ref{corollary: Bruhat Right} to obtain a representative of the right orbit $A \cdot \Gamma_{\infty}(3)$. By using equation \eqref{equation: Upsilon Multiplication}, we can find a representative of the left orbit $\Gamma_{\infty}(3) \cdot A$. 

\begin{theorem} \label{theorem: Bruhat Left 3 Specific}
The set of matrices $SL_3(\textgoth{o})$ is the disjoint union of the three cell decompositions

\begin{equation} \label{equation: Left Delta 2}
\bigsqcup_{y \in Y(\textgoth{o})} \bigsqcup_{d \in D(3)} \bigsqcup_{u \in U(3)} \Gamma_{\infty}(3) u d \varphi_1(y),
\end{equation}

\begin{equation} \label{equation: Left Delta 1 0}
\bigsqcup_{\substack{y_1, y_2 \in Y(\textgoth{o}), \\ y_2 \neq I_2}} \bigsqcup_{d \in D(3)} \bigsqcup_{u \in U(3)}  \Gamma_{\infty}(3) u d \varphi_2(y_2) \varphi_1(y_1), 
\end{equation}
and
\begin{equation} \label{equation: Left Delta 1 1}
\bigsqcup_{\substack{y_1, y_2, y_3 \in Y(\textgoth{o}), \\ y_2, y_3 \neq I_2}} \bigsqcup_{d \in D(3)} \bigsqcup_{u \in U(3)} \Gamma_{\infty}(3) u d \varphi_1(y_3) \varphi_2(y_2) \varphi_1(y_1). 
\end{equation}
\end{theorem}

\begin{proof}
Assume that 

\begin{equation*}
A = 
\begin{pmatrix}
a & b & c \\
d & e & f \\
g & h & i 
\end{pmatrix} \in SL_3(\textgoth{o}).
\end{equation*}
By equation \eqref{equation: Upsilon}, $A = \Upsilon^1(\Upsilon^1(A))$. Using Theorem \ref{theorem: SL3 Existence}, 

\begin{equation*}
\Upsilon^1(A) = \varphi_2(y_1^{-1}) \varphi_1(y_2^{-1}) \varphi_2(y_3^{-1}) d u C
\end{equation*}
where $y_1, y_2, y_3 \in Y(\textgoth{o})$, $d \in D(3)$, $u \in U(3)$ and $C \in \Gamma_{\infty}(3)$. Applying equation \eqref{equation: Upsilon Multiplication}, we find that 

\begin{equation*}
A = \Upsilon^1(\Upsilon^1(A)) = \Upsilon^1(C) \Upsilon^1(u) \Upsilon^1(d) \varphi_1(y_3) \varphi_2(y_2) \varphi_1(y_1)
\end{equation*}
where $\Upsilon^1(C) \in \Gamma_{\infty}(3)$ and $\Upsilon^1(d) \in D(3)$. Now assume that 

\begin{equation*}
u = 
\begin{pmatrix} 
1 & x & y \\
& 1 & z \\
& & 1 
\end{pmatrix} \in U(3) \qquad \text{so that} \qquad 
\Upsilon^1(u) = 
\begin{pmatrix} 
1 & -z & y \\
& 1 & -x \\
& & 1 
\end{pmatrix}.
\end{equation*}
Recall the map $(-)_3$ from equation \eqref{equation: Mod 3 Map}. Choose unique elements $m, n, p \in \textgoth{o}$ such that 

\begin{equation*}
-z - 3m = (-z)_3, \qquad y - 3n = (y)_3 \qquad \text{and} \qquad -x - 3p = (x)_3. 
\end{equation*} 
Then, $\Upsilon^1(u) = KL$ where

\begin{equation*}
K = 
\begin{pmatrix}
1 & 3m & 3mx + 9mp + 3n \\
& 1 & 3p \\
& & 1
\end{pmatrix} \in \Gamma_{\infty}(3) \qquad \text{and} \qquad L =
\begin{pmatrix}
1 & -z - 3m & y - 3n \\
& 1 & -x - 3p \\
& & 1
\end{pmatrix} \in U(3).
\end{equation*}
Hence, 

\begin{equation*}
A = \Upsilon^1(\Upsilon^1(A)) = (\Upsilon^1(C) K) L \Upsilon^1(d) \varphi_1(y_3) \varphi_2(y_2) \varphi_1(y_1).
\end{equation*}
By Corollary \ref{corollary: Bruhat Right} and the fact that the map $\Upsilon^1: M_{3 \times 3}(\textgoth{o}) \rightarrow M_{3 \times 3}(\textgoth{o})$ is bijective, we have

\begin{equation*}
SL_3(\textgoth{o}) = \Upsilon^1(SL_3(\textgoth{o})) = \Upsilon^1(\Delta_2) \sqcup \Upsilon^1(\Delta_{1, 0}) \sqcup \Upsilon^1(\Delta_{1, 1}) .
\end{equation*}
Thus by specialising the above construction to the disjoint subsets $\Delta_2, \Delta_{1, 0}$ and $\Delta_{1, 1}$ of $SL_3(\textgoth{o})$, we obtain the desired cell decomposition of $SL_3(\textgoth{o})$. 
\end{proof}

\begin{remark}
\emph{We can use the inverse $A^{-1}$ in place of $\Upsilon^1(A)$ to prove an analogous result to Theorem \ref{theorem: Bruhat Left 3 Specific}. One reason why $\Upsilon^1(A)$ was used in the above proof is because by equation \eqref{equation: Upsilon}, $\Upsilon^1(A)$ is much easier to compute than $A^{-1}$.} \newline

\emph{The proof of Theorem \ref{theorem: Bruhat Left 3 Specific} relies on decomposing $\Upsilon^1(A)$ by using Theorem \ref{theorem: SL3 Existence}. In the right orbit $\Upsilon^1(A) \cdot \Gamma_{\infty}(3)$, the invariants are given by the first columns of $\Upsilon^1(A)$ and $\Lambda^2(\Upsilon^1(A))$, which are, by equation \eqref{equation: Upsilon},}

\begin{equation*}
[i, -h, g]^T \qquad \text{and} \qquad [ei - fh, fg - di, dh - eg]^T
\end{equation*}
\emph{respectively. The elements of both columns are the $\Lambda^1$ and $\Lambda^2$ invariants of $A$ up to a few signs. This strong connection to the $\Lambda^1$ and $\Lambda^2$ invariants of $A$ is another reason why we used $\Upsilon^1(A)$ in the proof of Theorem \ref{theorem: Bruhat Left 3 Specific} rather than $A^{-1}$.}
\end{remark}

\begin{example}
As an example illustrating Theorem \ref{theorem: Bruhat Left 3 Specific}, let 

\begin{equation*}
A = 
\begin{pmatrix}
4 & -3 & -12 \\
-3 & 4 & 15 \\
-6 & 3 & 13 \\
\end{pmatrix} \in \Gamma(3). 
\end{equation*}
By following the steps outlined in the proof of Theorem \ref{theorem: Bruhat Left 3 Specific}, $A = C \cdot u \cdot d \cdot \varphi_1(y_3) \cdot \varphi_2(y_2) \cdot \varphi_1(y_1)$, where

\begin{equation*}
y_1 = 
\begin{pmatrix}
1 & -1 \\
2 & -1
\end{pmatrix}, \qquad
y_2 = 
\begin{pmatrix}
2 & -9 \\
3 & -13
\end{pmatrix}, \qquad
y_3 = 
\begin{pmatrix}
2 & 23 \\
5 & 58
\end{pmatrix} \in Y(\textgoth{o}),
\end{equation*}

\begin{equation*}
d = \begin{pmatrix}
1 & & \\
& - 1 & \\
& & -1 \\
\end{pmatrix} \in D(3), \qquad u = 
\begin{pmatrix}
1 & & \\
& 1 & \\
& & 1 \\
\end{pmatrix} \in U(3) 
\end{equation*}

\begin{equation*}
\text{and} \qquad
C = 
\begin{pmatrix}
1 & 0 & 15 \\
& 1 & -39 \\
& & 1 \\
\end{pmatrix} \in \Gamma_{\infty}(3). 
\end{equation*}
Consequently, $u \cdot d \cdot \varphi_1(y_3) \cdot \varphi_2(y_2) \cdot \varphi_1(y_1)$ is a representative of the left orbit $\Gamma_{\infty}(3) \cdot A$ belonging to the cell decomposition in equation \eqref{equation: Left Delta 1 1}.
\end{example}

\subsection{Proof of the main theorem}

This section is dedicated to the statement and proof of the main theorem, which begins with a sequence $S = (A_1, B_1, C_1, A_2, B_2, C_2) \in \textgoth{o}^6$ satisfying  invariant conditions \eqref{equation: I1}, \eqref{equation: I2}, \eqref{equation: I3},  \eqref{equation: I4} and produces a matrix $A \in \Gamma(3)$ such that $Inv(A) = S$. The matrix $A$ decomposes according to Theorem \ref{theorem: Bruhat Left 3 Specific} and in turn, is a matrix representative of the orbit in $\Gamma_{\infty}(3) \backslash \Gamma(3)$ whose invariants are given by $S$. To reiterate, there are five different cases to consider, which are outlined in \eqref{equation: Cases}. 

\begin{theorem} \label{theorem: Bruhat as Invariants + Cases}
Let $(A_1, B_1, C_1, A_2, B_2, C_2) \in \textgoth{o}^6$ satisfy the invariant conditions \eqref{equation: I1}, \eqref{equation: I2}, \eqref{equation: I3} and \eqref{equation: I4}. \newline

Case 1: If $A_1 \neq 0$, $A_2 \neq 0$ and $i \in \{1, 2, 3 \}$ then define $p_i, q_i, s_i \in \textgoth{o}$, $r_i \in \textgoth{o}/\textgoth{o}^{\times}$ and $\alpha, \beta, \gamma \in \textgoth{o}^{\times}$ such that 

\begin{equation} \label{equation: Case 1}
\begin{matrix}
A_1 = \gcd(A_1, B_1) r_1, & & p_1 \in \textgoth{o}/r_1 \textgoth{o}, & & \gcd(A_1, B_1) s_1 = B_1, & & p_1 s_1 - q_1 r_1 = 1, \\
\\
\gcd(A_1, B_1) = \alpha r_2, & & p_2 \in \textgoth{o}/r_2 \textgoth{o}, & & s_2 = \alpha^{-1} C_1, & & p_2 s_2 - q_2 s_2 = 1, \\
\\
A_2 = \gcd(A_1, B_1) \beta r_3, & & p_3 \in \textgoth{o}/r_3 \textgoth{o}, & & s_3 = \gamma(p_1 C_2 - q_1 B_2), & & p_3 s_3 - q_3 r_3 = 1, \\
\\
\alpha \beta \gamma = 1.
\end{matrix}
\end{equation}

Case 2: If $A_1 \neq 0$, $A_2 = 0$ and $i \in \{1, 2, 3 \}$ then define $p_i, q_i, s_i \in \textgoth{o}$, $r_i \in \textgoth{o}/\textgoth{o}^{\times}$ and $\alpha, \beta, \gamma \in \textgoth{o}^{\times}$ such that 

\begin{equation} \label{equation: Case 2}
\begin{matrix}
A_1 = \gcd(A_1, B_1) r_1, & & p_1 \in \textgoth{o}/r_1 \textgoth{o}, & & \gcd(A_1, B_1) s_1 = B_1, & & p_1 s_1 - q_1 r_1 = 1, \\
\\
\gcd(A_1, B_1) = \alpha r_2, & & p_2 \in \textgoth{o}/r_2 \textgoth{o}, & & s_2 = \alpha^{-1} C_1, & & p_2 s_2 - q_2 s_2 = 1, \\
\\
r_3 = 0, & & p_3 = 1, & & s_3 = 1, & & q_3 = 0, \\
\\
p_1 C_2 - q_1 B_2 = \alpha \beta, & & \alpha \beta \gamma = 1.
\end{matrix}
\end{equation}

Case 3: If $A_1 = 0$, $A_2 \neq 0$, $B_1 \neq 0$ and $i \in \{1, 2, 3 \}$ then define $p_i, q_i, s_i \in \textgoth{o}$, $r_i \in \textgoth{o}/\textgoth{o}^{\times}$ and $\alpha, \beta, \gamma \in \textgoth{o}^{\times}$ such that 

\begin{equation} \label{equation: Case 3}
\begin{matrix}
r_1 = 0, & & p_1 = 1, & & s_1 = 1, & & q_1 = 0, \\
\\
B_1 = \alpha r_2, & & p_2 \in \textgoth{o}/r_2 \textgoth{o}, & & s_2 = \alpha^{-1} C_1, & & p_2 s_2 - q_2 s_2 = 1, \\
\\
A_2 = \beta B_1 r_3, & & p_3 \in \textgoth{o}/r_3 \textgoth{o}, & & s_3 = \gamma C_2, & & p_3 s_3 - q_3 r_3 = 1, \\
\\
\alpha \beta \gamma = 1.
\end{matrix}
\end{equation}

Case 4: If $A_1 = A_2 = 0$, $B_1 \neq 0$ and $i \in \{1, 2, 3 \}$ then define $p_i, q_i, s_i \in \textgoth{o}$, $r_i \in \textgoth{o}/\textgoth{o}^{\times}$ and $\alpha, \beta, \gamma \in \textgoth{o}^{\times}$ such that 

\begin{equation} \label{equation: Case 4}
\begin{matrix}
r_1 = 0, & & p_1 = 1, & & s_1 = 1, & & q_1 = 0, \\
\\
B_1 = \alpha r_2, & & p_2 \in \textgoth{o}/r_2 \textgoth{o}, & & s_2 = \beta C_1, & & p_2 s_2 - q_2 s_2 = 1, \\
\\
r_3 = 0, & & p_3 = 1, & & s_3 = 1, & & q_3 = 0, \\
\\
\alpha \beta = 1, & & \gamma = 1.
\end{matrix}
\end{equation}
In each of these cases, define the matrices

\begin{equation} \label{equation: D And Y}
d = 
\begin{pmatrix}
\gamma & & \\
& \beta & \\
& & \alpha
\end{pmatrix}, \; y_1 =
\begin{pmatrix}
p_1 & q_1 \\
r_1 & s_1
\end{pmatrix}, \; y_2 = 
\begin{pmatrix}
p_2 & q_2 \\
r_2 & s_2
\end{pmatrix}, \; y_3 = 
\begin{pmatrix}
p_3 & q_3 \\
r_3 & s_3
\end{pmatrix}.
\end{equation}
Next, define

\begin{equation} \label{equation: W And V}
W = d \cdot \varphi_1(y_3) \cdot \varphi_2(y_2) \cdot \varphi_1(y_1), \qquad W^{-1} = (z_{ij}), \qquad  V = \big( (z_{ij})_3 \big)
\end{equation}

where in the definition of $V$, we recall the map $(-)_3$ from equation \eqref{equation: Mod 3 Map}. Then, the matrix

\begin{equation*}
X = VW \in \Gamma(3) \qquad \text{and satisfies} \qquad Inv(X) = (A_1, B_1, C_1, A_2, B_2, C_2).
\end{equation*}
Moreover, the decomposition $X = I_3 \cdot V \cdot d \cdot \varphi_1(y_3) \cdot \varphi_2(y_2) \cdot \varphi_1(y_1)$ with the identity matrix $I_3 \in \Gamma_{\infty}(3)$ is the decomposition from Theorem \ref{theorem: Bruhat Left 3 Specific}. \newline

Case 5: If $A_1 = B_1 = A_2 = 0$ then the matrix

\begin{equation} \label{equation: Corollary X}
X = 
\varphi_1
\begin{pmatrix}
a - b B_2  & b - b C_2  \\
B_2 & C_2 \\
\end{pmatrix} \qquad \text{with} \qquad a C_2 - b B_2 = 1
\end{equation}
is an element of $\Gamma(3)$ which satisfies $Inv(X) = (A_1, B_1, C_1, A_2, B_2, C_2)$.
\end{theorem} 

\begin{proof}
We will first prove case 1. Assume that the sequence $(A_1, B_1, C_1, A_2, B_2, C_2) \in \textgoth{o}^6$ satisfies the invariant conditions \eqref{equation: I1}, \eqref{equation: I2}, \eqref{equation: I3} and \eqref{equation: I4}. Assume that $A_1 \neq 0$ and $A_2 \neq 0$. Then, $\gcd(A_1, B_1) \neq 0$ which ensures that we can select $r_1, r_2 \in \textgoth{o} - \{ 0 \}/\textgoth{o}^{\times}$ and $\alpha \in \textgoth{o}^{\times}$ satisfying the equations

\begin{equation*}
\gcd(A_1, B_1) = \alpha r_2 \qquad \text{and} \qquad A_1 = \gcd(A_1, B_1) r_1.
\end{equation*}

Next, we will show that $\gcd(A_1, B_1)$ divides $A_2$. By invariant condition \eqref{equation: I3}, $\gcd(A_1, B_1, C_1) = 1$. So, there exists $\mu, \delta, \epsilon \in \textgoth{o}$ such that $\mu A_1 + \delta B_1 + \epsilon C_1 = 1$. Multiplying both sides by $A_2$, we obtain
\begin{IEEEeqnarray*}{rCl}
A_2 & = & \mu A_1 A_2 + \delta B_1 A_2 + \epsilon A_2 C_1
\\
& = & \mu A_1 A_2 + \delta B_1 A_2 + \epsilon (B_1 B_2 - A_1 C_2)
\end{IEEEeqnarray*} 
by invariant condition \eqref{equation: I4}. So, $A_2 \in A_1 \textgoth{o} + B_1 \textgoth{o} = \gcd(A_1, B_1) \textgoth{o}$ and $\gcd(A_1, B_1)$ divides $A_2$. Hence,  we can select $\beta, \gamma \in \textgoth{o}^{\times}$ and $r_3 \in \textgoth{o} - \{ 0 \}/\textgoth{o}^{\times}$ such that 

\begin{equation*}
A_2 = \gcd(A_1, B_1) \beta r_3 \qquad \text{and} \qquad \alpha \beta \gamma = 1. 
\end{equation*}

Now we begin our construction of the matrix $X$ by first constructing the matrices $y_1, y_2, y_3$ in equation \eqref{equation: D And Y} so that $y_1, y_2, y_3 \in Y(\textgoth{o})$ (see equation \eqref{equation: YR}). Let $s_1, s_2 \in \textgoth{o}$ satisfy

\begin{equation*}
\gcd(A_1, B_1) s_1 = B_1 \qquad \text{and} \qquad s_2 = \alpha^{-1} C_1. 
\end{equation*}
Beginning with $y_1$, we observe that 
\begin{equation*}
\gcd(A_1, B_1) = \gcd(\gcd(A_1, B_1) r_1, \gcd(A_1, B_1) s_1) = \gcd(A_1, B_1) \gcd(r_1, s_1). 
\end{equation*}
Since $\gcd(A_1, B_1) \neq 0$ then $\gcd(r_1, s_1) = 1$. So, there exists $p_{1, 0}, q_{1, 0} \in \textgoth{o}$ such that $p_{1, 0} s_1 - q_{1, 0} r_1 = 1$. If $n \in \textgoth{o}$ then 

\begin{equation*}
p_{1, n} = p_{1, 0} - n r_1 \qquad \text{and} \qquad q_{1, n} = q_{1, 0} - n s_1
\end{equation*}
satisfy $p_{1, n} s_1 - q_{1, n} r_1 = 1$. Thus, there exists $\ell \in \textgoth{o}$ such that $p_{1, \ell} \in \textgoth{o}/r_1 \textgoth{o}$. Letting $p_1 = p_{1, \ell}$ and $q_1 = q_{1, \ell}$, we obtain $p_1 s_1 - q_1 r_1 = 1$ with $p_1 \in \textgoth{o}/r_1 \textgoth{o}$. Consequently, $y_1 \in Y(\textgoth{o})$. \newline

Next, to see that $\gcd(r_2, s_2) = 1$, observe that by invariant condition \eqref{equation: I3}, $\gcd(A_1, B_1, C_1) = 1$ and 

\begin{equation*}
\gcd(r_2, s_2) = \gcd(\gcd(A_1, B_1), C_1)  = \gcd(A_1, B_1, C_1) = 1.
\end{equation*}

Similarly, there exists $p_2 \in \textgoth{o}/r_2 \textgoth{o}$ and $q_2 \in \textgoth{o}$ such that $p_2 s_2 - q_2 r_2 = 1$. So, $y_2 \in Y(\textgoth{o})$. \newline

Finally, let $s_3 = \gamma (p_1 C_2 - q_1 B_2)$ as in equation \eqref{equation: Case 1}. To see that $\gcd(r_3, s_3) = 1$, we will find $u, v \in \textgoth{o}$ such that $ur_3 + v s_3 = 1$. Since $\gcd(A_2, B_2, C_2) = 1$ by invariant condition \eqref{equation: I3}, there exists $\delta_1, \delta_2, \delta_3 \in \textgoth{o}$ such that $\delta_1 A_2 + \delta_2 B_2 + \delta_3 C_2 = 1$. \newline

We will now derive expressions for $B_2$ and $C_2$. First, observe that

\begin{equation} \label{equation: 4.4}
p_1 B_1 - q_1 A_1 = \gcd(A_1, B_1)(p_1 s_1 - q_1 r_1) = \gcd(A_1, B_1) = \alpha r_2.
\end{equation}
and
\begin{equation} \label{equation: 4.5}
p_1 C_2 - q_1 B_2 = \gamma^{-1} s_3 = \alpha \beta s_3.
\end{equation}
If we multiply equation \eqref{equation: 4.4} by $C_2$ and then subtract equation \eqref{equation: 4.5} multiplied by $B_1$ from it, we find that since $A_1 C_2 - B_1 B_2 + A_2 C_1 = 0$, 

\begin{equation*}
q_1 A_2 C_1 = q_1(B_1 B_2 - A_1 C_2) = \alpha r_2(C_2 - \alpha \beta s_1s_3).
\end{equation*}
By the equations $A_2 = \alpha \beta r_2 r_3$ and $C_1 = \alpha s_2$, we find at

\begin{equation} \label{equation: C2}
C_2 = \alpha \beta (q_1 r_3 s_2 + s_1 s_3).
\end{equation}

By a similar argument,

\begin{equation} \label{equation: B2}
B_2 = \alpha \beta(p_1 r_3 s_2 + r_1 s_3). 
\end{equation}
Therefore, if we substitute equations \eqref{equation: C2}, \eqref{equation: B2} and $A_2 = \alpha \beta r_2 r_3$ into the equation $\delta_1 A_2 + \delta_2 B_2 + \delta_3 C_2 = 1$, we obtain
\begin{equation*}
1 = r_3(\delta_1 \alpha \beta r_2 + \alpha \beta \delta_2 p_1 s_2 + \delta_3 \alpha \beta q_1 s_2) + s_3(\alpha \beta \delta_2 r_1 + \alpha \beta \delta_3 s_1). 
\end{equation*}
Hence, $r_3 \textgoth{o} + s_3 \textgoth{o} = \textgoth{o}$ and $\gcd(r_3, s_3) = 1$. Analogously to before, we obtain $p_3 \in \textgoth{o}/r_3 \textgoth{o}$ and $q_3 \in \textgoth{o}$ such that $p_3 s_3 - q_3 r_3 = 1$. We conclude that the matrix $y_3 \in Y(\textgoth{o})$. \newline

Next, assume that $d$ is defined as in equation \eqref{equation: D And Y}. Since $\gamma \beta \alpha = 1$, $d \in D(3)$. Now assume that $W$ and $V$ are defined as in \eqref{equation: W And V}. To see that $V \in U(3)$, we will first show that 

\begin{equation} \label{equation: Some Matrix}
W = 
\begin{pmatrix}
\gamma (p_1 p_3 + p_2 q_3 r_1) & \gamma (p_2 q_3 s_1 + p_3 q_1) & \gamma q_2 q_3 \\
\beta(p_1 r_3 + p_2 r_1 s_3) & \beta(p_2 s_1 s_3 + q_1 r_3) & \beta q_2 s_3 \\
\alpha r_1 r_2 & \alpha r_2 s_1 & \alpha s_2 \\
\end{pmatrix}
\end{equation}
is congruent to an upper triangular, unipotent matrix mod $3 \textgoth{o}$, where the congruence is computed entrywise. Let $w_{ij}$ denote the $i,j$ entry of $W$. For the bottom row of $W$, we note that $w_{31} = \alpha r_1 r_2 = A_1 \equiv 0$ mod $3 \textgoth{o}$, $w_{32} = \alpha r_2 s_1 = B_1 \equiv 0$ mod $3 \textgoth{o}$ and $w_{33} = \alpha s_2 = C_1 \equiv 1$ mod $3 \textgoth{o}$ by invariant condition \eqref{equation: I2}. \newline

Before we compute $w_{11}, w_{21}$ and $w_{22}$ modulo $3 \textgoth{o}$ we must first establish a few congruence relations. By invariant conditions \eqref{equation: I1} and \eqref{equation: I2}, $\gcd(A_1, B_1) \equiv 0$ mod $3 \textgoth{o}$, 

\begin{equation} \label{equation: Mod 1}
r_2 = \alpha^{-1} \gcd(A_1, B_1) \equiv 0 \text{ mod } 3 \textgoth{o},
\end{equation}
\begin{equation} \label{equation: Mod 2}
s_2 = \alpha^{-1} C_1 \equiv \alpha^{-1} \text{ mod } 3 \textgoth{o}
\end{equation}
and
\begin{equation} \label{equation: Mod 3}
s_3 \equiv \gamma (p_1(1) - q_1(0)) \equiv \gamma p_1 \text{ mod } 3 \textgoth{o}.
\end{equation}

In turn, equations \eqref{equation: Mod 1} and \eqref{equation: Mod 2} yield

\begin{equation} \label{equation: Mod 4}
p_2 = (p_2 \alpha^{-1}) \alpha \equiv p_2 s_2 \alpha \equiv (1 + q_2 r_2) \alpha \equiv \alpha \text{ mod } 3 \textgoth{o}. 
\end{equation}
By invariant condition \eqref{equation: I4}, $A_1 C_2 - B_1 B_2 + A_2 C_1 = 0$ and $r_1 C_2 - s_1 B_2 + \beta r_3 C_1 = 0$. Reducing modulo $3 \textgoth{o}$, we find that 

\begin{equation*}
0 = r_1 C_2 - s_1 B_2 + \beta r_3 C_1 \equiv r_1(1) - s_1(0) + \beta r_3(1) \equiv r_1 + \beta r_3 \text{ mod } 3 \textgoth{o},
\end{equation*}
and
\begin{equation} \label{equation: Mod 5}
r_1 \equiv - \beta r_3 \text{ mod } 3 \textgoth{o}. 
\end{equation}

Reducing the matrix elements $w_{11}$, $w_{21}$ and $w_{22}$ modulo $3 \textgoth{o}$, we obtain
\begin{IEEEeqnarray*}{rCl}
w_{21} & = & \beta r_3 p_1 + \beta p_2 r_1 s_3
\\
& \equiv & (- r_1) p_1 +  \beta p_2 r_1 s_3 \text{ mod } 3\textgoth{o} \quad \eqref{equation: Mod 5}
\\
& \equiv & - p_1 r_1 + \beta \alpha r_1 s_3 \text{ mod } 3\textgoth{o} \quad \eqref{equation: Mod 4}
\\
& \equiv & - p_1 r_1 + \beta \alpha r_1 (\gamma p_1) \text{ mod } 3\textgoth{o} \quad \eqref{equation: Mod 3}
\\
& \equiv & 0 \text{ mod } 3\textgoth{o},
\end{IEEEeqnarray*}
\begin{IEEEeqnarray*}{rCl}
w_{22} & = & \beta p_2 s_1 s_3 + \beta q_1 r_3
\\
& \equiv & \beta p_2 s_1 s_3 + q_1 (-r_1) \text{ mod } 3\textgoth{o} \quad \eqref{equation: Mod 5}
\\
& \equiv & \beta \alpha s_1 s_3 - q_1 r_1 \text{ mod } 3\textgoth{o} \quad \eqref{equation: Mod 4}
\\
& \equiv & \beta \alpha s_1 (\gamma p_1) - q_1 r_1 \text{ mod } 3\textgoth{o} \quad \eqref{equation: Mod 3}
\\
& \equiv & p_1 s_1 - q_1 r_1 \equiv 1 \text{ mod } 3\textgoth{o},
\end{IEEEeqnarray*}
\begin{IEEEeqnarray*}{rCl}
w_{11} & = & \gamma p_1 p_3 + \gamma  p_2 q_3 r_1
\\
& \equiv & s_3 p_3 + \gamma p_2 q_3 r_1 \text{ mod } 3\textgoth{o} \quad \eqref{equation: Mod 3}
\\
& \equiv & s_3 p_3 + \gamma  p_2 q_3 (- \beta r_3) \text{ mod } 3\textgoth{o} \quad \eqref{equation: Mod 5}
\\
& \equiv & p_3 s_3 - q_3 r_3 \text{ mod } 3\textgoth{o} \quad \eqref{equation: Mod 4}
\\
& \equiv & 1 \text{ mod } 3\textgoth{o}.
\end{IEEEeqnarray*}
Hence, 

\begin{equation*}  
W \equiv
\begin{pmatrix}
1 & \gamma (p_2 q_3 s_1 + p_3 q_1) & \gamma q_2 q_3 \\
 & 1 & \beta q_2 s_3 \\
 &  & 1 \\
\end{pmatrix} \mod 3 \textgoth{o}
\end{equation*}
and so the matrix $W^{-1}$ is congruent to an upper triangular unipotent matrix mod $3 \textgoth{o}$. Since the matrix $V$ is obtained by reducing each entry of $W^{-1}$ mod $3 \textgoth{o}$ then $V \in U(3)$. \newline

Now suppose that $X = VW$. Since $y_1, y_2, y_3 \in Y(\textgoth{o})$, $d \in D(3)$ and $V \in U(3)$, then the decomposition $X = I_3 \cdot V \cdot d \cdot \varphi_1(y_3) \cdot \varphi_2(y_2) \cdot \varphi_1(y_1)$ is the same one in Theorem \ref{theorem: Bruhat Left 3 Specific}. In particular, since $y_2, y_3 \neq I_2$, it is the decomposition from equation \eqref{equation: Left Delta 1 1}. The matrix $X \in \Gamma(3)$ because $X = VW \equiv W^{-1} W \equiv I_3$ mod $3 \textgoth{o}$ and $\det(X) = 1$ by Theorem \ref{theorem: Bruhat Left 3 Specific}. \newline

To check that $Inv(X) = (A_1, B_1, C_1, A_2, B_2, C_2)$ it suffices to compute $Inv(W)$ because since $V \in U(3)$ then $Inv(W) = Inv(VW) = Inv(X)$. Recalling the definition of $W$ from equation \eqref{equation: Some Matrix}, the $\Lambda^1$ invariants of this matrix are $\alpha r_1 r_2 = A_1$, $\alpha r_2 s_1 = B_1$ and $\alpha s_2 = C_1$. The $\Lambda^2$ invariants of $W$, which are the entries of the bottom row of $\Lambda^2(W)$, satisfy 
\begin{align*}
\alpha \beta r_2 s_1(p_1 r_3 + p_2 r_1 s_3) & - \alpha \beta r_1 r_2 (p_2 s_1 s_3 + q_1 r_3)  = \alpha \beta p_1 r_2 r_3 s_1 - \alpha \beta r_1 r_ 2 r_3 q_1 \\
& = \alpha \beta r_2 r_3(p_1 s_1 - q_1 r_1) \\
& = \alpha \beta r_2 r_3 = A_2,
\end{align*}
\begin{align*}
\gcd(A_1, B_1) (\alpha \beta s_2(p_1 r_3 + p_2 r_1 s_3) & - \alpha \beta q_2 s_3 r_1 r_2) = \gcd(A_1, B_1)(\alpha \beta p_1 r_3 s_2 + \alpha \beta r_1 s_3) \\
& = p_1 A_2 C_1 + p_1 A_1 C_2 - q_1 A_1 B_2 \\
& = p_1 B_1 B_2 - q_1 A_1 B_2 \quad \eqref{equation: I4} \\
& = \gcd(A_1, B_1) (p_1 s_1 B_2 - q_1 r_1 B_2) = \gcd(A_1, B_1) B_2,
\end{align*}
and
\begin{align*}
\gcd(A_1, B_1)(\alpha \beta s_2(p_2 s_1 s_3 + q_1 r_3) & - \alpha \beta  q_2 r_2 s_1 s_3) = \gcd(A_1, B_1)(\alpha \beta  s_1 s_3 + \alpha \beta  q_1 r_3 s_2) \\
& = p_1 B_1 C_2 - q_1 B_1 B_2 + q_1 A_2 C_1 \\
& = p_1 B_1 C_2 - q_1 A_1 C_2 \quad \eqref{equation: I4} \\
& = \gcd(A_1, B_1)( p_1 s_1 C_2 - q_1 r_1 C_2) = \gcd(A_1, B_1) C_2.
\end{align*}
Since $\gcd(A_1, B_1) \neq 0$ then the $\Lambda^2$ invariants of $W$ are $A_2, B_2$ and $C_2$. Hence, 

\begin{equation*}
(A_1, B_1, C_1, A_2, B_2, C_2) = Inv(W) = Inv(X). 
\end{equation*}
This completes the proof of case 1. \newline

The proofs of cases 2, 3 and 4 in \eqref{equation: Cases} are similar to case 1. The main differences are that in case 2, $s_3 = 1$, $p_1 C_2 - q_1 B_2 = \alpha \beta$ and the resulting decomposition is the one from equation \eqref{equation: Left Delta 1 0}. In case 3, $A_1 = 0$ and $B_1 \neq 0$ which means that $\gcd(A_1, B_1) = B_1$,

\begin{equation*}
B_1 = \alpha r_2, \qquad A_2 = \beta B_1 r_3 \qquad \text{and} \qquad s_3 = \gamma C_2. 
\end{equation*}

In this case, the resulting decomposition is the one from equation \eqref{equation: Left Delta 1 1}. Case 4 is a combination of cases 2 and 3. In case 4, $\beta = \alpha^{-1}$ and the decomposition $X = I_3 V d \varphi_2(y_2)$ is the one from equation \eqref{equation: Left Delta 1 0}. These changes to the proof of case 1 provide an outline of what is required to complete the proofs of cases 2, 3 and 4. \newline

For case 5, assume that $A_1 = B_1 = A_2 = 0$. By invariant condition \eqref{equation: I2}, $C_1 \equiv 1$ mod $3 \textgoth{o}$ and by invariant condition \eqref{equation: I3}, $\gcd(0, 0, C_1) = 1$ and $\gcd(0, B_2, C_2) = \gcd(B_2, C_2) = 1$. Consequently, $C_1 = 1$ and there exists $a, b \in \textgoth{o}$ such that $a C_2 - b B_2 = 1$. \newline

If $m \in \textgoth{o}$ then define $a_m = a - m B_2$ and $b_m = b - m C_2$. Then, $a_m C_2 - b_m B_2 = 1$. By setting $m = b$ we obtain

\begin{equation*}
a_{b} = a - b B_2 \qquad \text{and} \qquad b_{b} = b - b C_2. 
\end{equation*}
Since $B_2 \equiv 0$ mod $3\textgoth{o}$ and $C_2 \equiv 1$ mod $3\textgoth{o}$ from invariant conditions \eqref{equation: I1} and \eqref{equation: I2}, $b_{b} \equiv 0$ mod $3\textgoth{o}$. Since $a_{b} C_2 - b_{b} B_2 = 1$ then

\begin{equation*}
1 = a_b C_2 - b_b B_2 \equiv a_b (1) - (0)(0) \equiv a_b \text{ mod } 3 \textgoth{o}.
\end{equation*}
By defining the matrix $X$ as in equation \eqref{equation: Corollary X}, we find, by direct computation, that $X \in \Gamma(3)$, the decomposition of $X$ is the one in equation \eqref{equation: Left Delta 2} and $Inv(X) = (0, 0, C_1, A_2, B_2, C_2)$ as required.
\end{proof}

\begin{example}
Here is an application of Theorem \ref{theorem: Bruhat as Invariants + Cases} to a concrete example. Suppose that we have the following elements of $\textgoth{o}$ which satisfy the invariant conditions:

\begin{equation*}
\begin{matrix}
A_1 = -3 + 6 \omega, & B_1 = -3, & C_1 = -2 - 3 \omega, \\
A_2 = -6 + 3 \omega, & B_2 = 3 - 6 \omega, & C_2 = 4 + 3 \omega.
\end{matrix}
\end{equation*}
By applying Theorem \ref{theorem: Bruhat as Invariants + Cases}, the matrix

\begin{equation*}
X =
\begin{pmatrix}
-11 - 3 \omega & -3 - 3 \omega & -3 \omega \\
-24 - 33 \omega & - 2 - 12 \omega & 12 + 3 \omega \\
-3 + 6 \omega & -3 & -2 - 3 \omega \\
\end{pmatrix} \in \Gamma(3)
\end{equation*}
satisfies $Inv(X) = (A_1, B_1, C_1, A_2, B_2, C_2)$. Thus, $X$ is a representative of the matrix orbit $\Gamma_{\infty}(3) \cdot A$, where $A \in \Gamma(3)$ and $Inv(A) = (A_1, B_1, C_1, A_2, B_2, C_2)$. Moreover, 

\begin{equation} \label{equation: X Decomp Example}
X =
\begin{pmatrix}
1 & 0 & 2 \omega \\
& 1 & 2 + \omega \\
& & 1 
\end{pmatrix}
\begin{pmatrix}
-1 - \omega & & \\
& 1 + \omega & \\
& & 1 + \omega
\end{pmatrix} \varphi_1(y_3) \varphi_2(y_2) \varphi_1(y_1),
\end{equation}
where

\begin{equation*}
y_1 = 
\begin{pmatrix}
1 + 2 \omega & \omega \\
2 + 3 \omega & \omega
\end{pmatrix}, \quad 
y_2 =
\begin{pmatrix}
1 + \omega & -1 - \omega \\
3 & -3 - \omega
\end{pmatrix} \quad \text{and} \quad
y_3 = 
\begin{pmatrix}
2 + 2 \omega & 1 + 6 \omega \\
3 + 2 \omega & 4 + 8 \omega
\end{pmatrix}.
\end{equation*}
One can verify that the decomposition of $X$ in equation \eqref{equation: X Decomp Example} is the same one stemming from equation \eqref{equation: Left Delta 1 1}.
\end{example}

\end{document}